\documentclass[12pt,a4paper,twoside]{amsart}
\usepackage[english]{babel} 
\usepackage[utf8]{inputenc}
\usepackage{amsmath,amssymb,amsthm,amscd}
\usepackage{graphicx}
\usepackage{fancyhdr}
\usepackage{enumerate}
\usepackage{indentfirst}
\usepackage{geometry}
\usepackage{hyperref}
\usepackage[all]{xy}
\usepackage{calc,color}
\newcommand{\R}{\ensuremath{\mathbb{R}}}
\newcommand{\N}{\ensuremath{\mathbb{N}}} 
\newcommand{\LL}{\ensuremath{\mathcal{L}}}
\newcommand{\delim}[3]{\left#1 #3 \right#2}  
\newcommand{\pin}[1]{\delim{\langle}{\rangle}{#1}} 
\def \ds {\displaystyle}  
\def \tn {\textnormal}  
  
\theoremstyle{definition} 
\newtheorem{theo}{Theorem}[section] 
\newtheorem{lem}[theo]{Lemma} 
\newtheorem{cor}[theo]{Corollary}

\newtheorem{rem}[theo]{Remark} 
\title[Parabolic problems with localized large diffusion]{Parabolic equations with localized large diffusion: Rate of convergence of attractors}

\author[A. N. Carvalho]{Alexandre N. Carvalho$^\dag$}\thanks{$^\dag$Research partially supported by CNPq  	303929/2015-4, and FAPESP 2018/10997-6, Brazil}
\address[Alexandre N. Carvalho]{Departamento de Matem\'atica,
Instituto de Ci{\^e}ncias Matem{\'a}ticas e de Computa\c{c}{\~a}o,
Universidade de S{\~a}o Paulo Campus de S{\~a}o Carlos,  S{\~a}o Carlos SP, Brazil}
\email{andcarva@icmc.usp.br}

\author[L. Pires]{Leonardo Pires$^*$}
\address[Leonardo Pires]{Departamento de Matem\'atica e Estat\'istica,
Universidade Estadual de Ponta Grossa, Ponta Grossa PR, Brazil}
\email{lpires@uepg.br}

\begin{document}

\begin{abstract}
In this paper we study the asymptotic nonlinear dynamics of scalar semilinear parabolic problems reaction-diffusion type when the diffusion coefficient becomes large in a subregion which is interior to the domain. We obtain, under suitable assumptions, that the family of attractors behaves continuously and we exhibit the rate of convergence. An accurate description of localized large diffusion is necessary.
\end{abstract}

\maketitle

\section{Introduction}

Local spatial homogenization is a feature that appears in several physical phenomenona. It is often present in heat conduction in materials for which the heat may diffuse much more faster in some regions than in others (composite materials). 

Reaction-diffusion models for which the diffusivity varies considerably from one region to another have solutions that tend to become spatially homogeneous in the regions where the diffusivity is large. There has been many studies of mathematical models for which this property was exploited (see, for example, \cite{J.M.Arrieta2000a, Carbone2008} and \cite{Rodriguez-Bernal1998}). 

Also, in \cite{Carvalho1994} the authors considered a scalar parabolic problem where the diffusivity is large except in a neighborhood of a finite number of points where it becomes small (see also, \cite{Fusco1987}). There it was shown that the asymptotic behavior is described by a system of linearly coupled ordinary differential equations. The analysis in \cite{Carvalho1994} requires a detailed description of the transition between large and small diffusion. 

Inspired in these works, we formulate a prototype problem of localized large difusion in order to understand its effects on the continuity of global attractors. Specifying the details on how the diffusivity becomes large or converge we obtain, not only continuity of attractor but also the rate of their convergence.
Our formulation refines the results on continuity of attractors (without rate) obtained in  \cite{Carbone2008} where the  diffusivity becomes large in a subregion which is interior to the domain. We will need to make an accurate description of localized large diffusion as in \cite{Carvalho1994} to be able to establish the rate of convergence.

To better present the main ideas while avoiding excessive notation, we consider the one dimensional scalar case for which the diffusion is large only in a part of the domain, leaving the case where the diffusion is large in a finite number of parts of the domain implicit.  

Consider the scalar parabolic problem
\begin{equation}\label{perturbed_problem}
\begin{cases}
u^\varepsilon_t-(p_\varepsilon(x) u^\varepsilon_x)_x+\lambda u^\varepsilon = f(u^\varepsilon),\quad 0<x<1,\,\,t>0, \\
u^\varepsilon_x(0)=u^\varepsilon_x(1)=0,\quad t>0,\\
u^\varepsilon(0)=u^\varepsilon_0,
\end{cases}
\end{equation}
where $\varepsilon \in (0,\varepsilon_0]$ is a parameter ($0<\varepsilon_0\leq 1$),  $f\in C^2(\R)$ and $\lambda >0$. To describe the coefficients $p_\varepsilon$, let $0=x_0<x_1<x_2<x_3=1$ be a partition of the interval $\Omega=(0,1)$. We assume the diffusion is very large in the open interval $\Omega_0=(x_1,x_2)$ and converges uniformly to $p_0\in C^2(\Omega_1)$ in the $\Omega_1=[0,x_1]\cup [x_2,1]$ as $\varepsilon$ approaches to zero. More precisely, for $m_0>0$ and  $\varepsilon \in (0,\varepsilon_0]$, $p_\varepsilon \in C^2([0,1])$ satisfies the following conditions (see Fig \ref{diffusion}).  

\begin{equation}\label{estimate_diffusion}
\begin{cases}     
p_\varepsilon \overset{\varepsilon\to 0}\longrightarrow p_0,\quad \tn{uniformly in } \Omega_1,\\
p_\varepsilon(x)\geq \dfrac{1}{\varepsilon}\quad \tn{in } [x_1+\varepsilon,x_2-\varepsilon],\\
m_0\leq p_\varepsilon\quad \tn{in } \Omega\quad\tn{and, consequently,}\quad m_0\leq p_0\quad \tn{in }\Omega_1.
\end{cases}
\end{equation}

\begin{figure}[h]
\begin{center}
\includegraphics[scale=0.75]{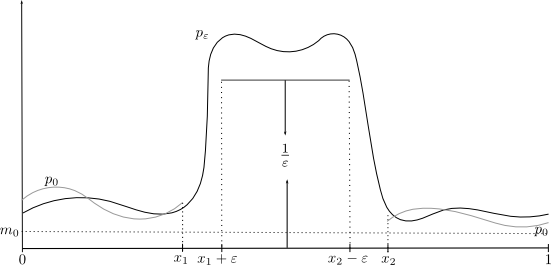} 
\caption{Diffusion}\label{diffusion}
\end{center}
\end{figure}

Due to the large diffusivity it is natural to expect that the solutions $u^\varepsilon$ of the \eqref{perturbed_problem} become approximately spatially constant on $\Omega_0$ as $t$ becomes large. 
Assume that  the  solutions $u^\varepsilon(t,x)$ exist and converge (in some sense), as $\varepsilon\to 0$, to a function $u^0(t,x)$ which is  spatially constant on $(x_1,x_2)$ and denote its value in $\Omega_0$ by $u^0_{\Omega_0}(t)$. With this, the limiting problem of \eqref{perturbed_problem} as $\varepsilon\to 0$ should be given by
\begin{equation}\label{limite_problem}
\begin{cases}
u^0_t-(p_0(x) u^0_x)_x+\lambda u^0 = f(u^0),\quad x\in \Omega_1,\,\,t>0, \\
u^0_x(0)=u^0_x(1)=0,\quad t>0,\\
u^0|_{\Omega_0}=u^0_{\Omega_0},\\
\dot{u}^0_{\Omega_0}+\lambda u^0_{\Omega_0}=f(u^0_{\Omega_0}),
\\
u^0(0)=u^0_0,
\end{cases}
\end{equation}
where $u_0^0=\lim_{\varepsilon\to 0} u^\varepsilon_0$ is constant on $\Omega_0$.

In fact, the problems \eqref{perturbed_problem} and \eqref{limite_problem} are particular cases of those in \cite{Rodriguez-Bernal1998} where it was proved that the solutions of the elliptic operator in \eqref{limite_problem} approximate those of the elliptic operator in \eqref{perturbed_problem} (resolvent convergence). As a consequence of that, the authors in \cite{Rodriguez-Bernal1998}, also proved the convergence of the spectrum of the associated elliptic operators. Nonetheless, nothing is said in \cite{Rodriguez-Bernal1998} about the rate of convergence resolvents or of the spectrum.

To better describe the results we write \eqref{perturbed_problem} and \eqref{limite_problem} abstractly in a natural energy space. To do that we introduce some more terminology. We start defining the operator $A_\varepsilon:\mathcal{D}(A_\varepsilon)\subset L^2(0,1)\to L^2(0,1)$ by
$$
\mathcal{D}(A_\varepsilon)=\{u\in H^2(0,1)\,;\, u_x(0)=u_x(1)=0\}\quad\tn{and}\quad A_\varepsilon u=-(p_\varepsilon u_x)_x+\lambda u.
$$
Next, we define the spaces of functions which are constant in $\Omega_0$ by
$$
L^2_{\Omega_0}(0,1)=\{u\in L^2(0,1)\,;\, u \tn{ is constant a.e. in } \Omega_0\} \ \hbox{ and }
$$
$$
H^1_{\Omega_0}(0,1)=\{u\in H^1(0,1)\,;\, u_x=0 \tn{ in } \Omega_0\}.
$$
Using this terminology, we define the operator $A_0:\mathcal{D}(A_0)\subset L^2_{\Omega_0}(0,1)\to L^2_{\Omega_0}(0,1)$ by
$$
\mathcal{D}(A_0)=\{u\in H^1_{\Omega_0}(0,1)\,;\,-(p_0 u_x)_x\in L^2(\Omega_1),\, u_x(0)=u_x(1)=0\};
$$
$$
A_0 u=\big[-(p_0 u_x)_x+\lambda u\big]\chi_{\Omega_1}+\Big[\lambda u_{\Omega_0}\Big]\chi_{\Omega_0},
$$
where $u_{\Omega_0}$ is the constant value of $u$ in $\Omega_0$.

Hence, $A_\varepsilon$ is a positive self-adjoint operator with compact resolvent for each $\varepsilon\in [0,\varepsilon_0]$, its first eigenvalue is $\lambda$ and we can define, in the usual way (see \cite{Henry1980}), the fractional power space $X_\varepsilon^\frac{1}{2}=H^1(0,1)$, $\varepsilon\in (0,\varepsilon_0]$, and $X_0^\frac{1}{2}=H^1_{\Omega_0}(0,1)$ with the respective scalar products 
$$
\pin{u,v}_{X_\varepsilon^{\frac{1}{2}}}=\int_0^1 p_\varepsilon u_xv_x\,dx+\int_0^1\lambda uv\,dx,\quad u,v\in X_\varepsilon^\frac{1}{2},\,\,\,\varepsilon\in (0,\varepsilon_0];
$$
$$
\pin{u,v}_{X_0^{\frac{1}{2}}}=\int_{\Omega_1} p_0 u_xv_x\,dx+\int_0^1\lambda uv\,dx,\quad u,v\in X_0^\frac{1}{2}.
$$

The space $X_0^\frac{1}{2}$ is a closed subspace of $X_\varepsilon^\frac{1}{2}$, $\varepsilon\in (0,\varepsilon_0]$ and $X_\varepsilon^\frac{1}{2}\subset H^1(0,1)$ with embedding constant independent of $\varepsilon$. However, a delicate issue here is that, the embedding $H^1(0,1)\subset X_\varepsilon^\frac{1}{2}$ is not independent of $\varepsilon$, in fact
$$
\|u\|_{H^1}\leq C\|u\|_{X_\varepsilon^\frac{1}{2}}\leq M_\varepsilon \|u\|_{H^1},\quad u\in H^1,
$$
where $M_\varepsilon\to \infty$ as $\varepsilon\to 0$.  Actually, it can be proved that it is not possible to choose $M_\varepsilon$ independent of $\epsilon$ (see \cite{CarvalhoPires2017}), therefore estimates in $H^1$ norm do not provide (uniform in $\epsilon$) estimates in $X_\varepsilon^\frac{1}{2}$ norm. 

We will consider $X_\varepsilon^\frac{1}{2}$ as the phase space for the problems \eqref{perturbed_problem} and \eqref{limite_problem} that is, if we also denote by $f$ the Nemytski\u{\i} operator associated to the function $f$, then \eqref{perturbed_problem} and \eqref{limite_problem} can be written in the following coupled form  
\begin{equation}\label{semilinear_problem}
\begin{cases}
u^\varepsilon_t+A_\varepsilon u^\varepsilon = f(u^\varepsilon), \\
u^\varepsilon(0)=u_0^\varepsilon \in X_\varepsilon^{\frac{1}{2}},\quad \varepsilon\in [0,\varepsilon_0].
\end{cases}
\end{equation}  
Since $X_\varepsilon^\frac{1}{2}\subset C([0,1])$, we do not require any growth condition for $f$ and we only assume that it satisfies the dissipativeness condition
$$
\limsup_{|x|\to\infty} \dfrac{f(x)}{x} <\lambda.
$$

It follows from \cite{J.M.Arrieta1999,J.M.Arrieta2000} and \cite{J.M.Arrieta2000a} that \eqref{semilinear_problem}, for each $\varepsilon\in [0,\varepsilon_0]$, is globally well posed and its solutions are classical and continuously differentiable with respect to the initial data. Also, we may assume, without loss of generality, that $f$ is globally bounded with globally bounded derivatives up to second order. Thus we are able to consider in $X_\varepsilon^{\frac{1}{2}}$ the family of nonlinear semigroups $\{T_\varepsilon(\cdot)\}_{\varepsilon\in [0,\varepsilon_0]}$ defined by $T_\varepsilon(t)=u^\varepsilon(t,u^\varepsilon_0)$, $t\geq 0$, where $u^\varepsilon(t,u^\varepsilon_0)$ is the solution of \eqref{semilinear_problem} through $u^\varepsilon_0\in X_\varepsilon^\frac{1}{2}$ and 
\begin{equation}\label{semigroupononlinear}
T_\varepsilon(t)u_0^\varepsilon   = e^{-A_{\varepsilon}t}u_0^\varepsilon+\int_0^{t} e^{-A_\varepsilon (t-s)}f(T_\varepsilon(s))\,ds,\quad t\geq 0, 
\end{equation}
has a global attractor $\mathcal{A}_\varepsilon$, for each $\varepsilon\in [0,\varepsilon_0]$ such that $\overline{\bigcup_{\varepsilon\in [0,\varepsilon_0]}\mathcal{A}_\varepsilon}$ is compact in $H^1(0,1)$.  

We recall that the family $\{\mathcal{A}_\varepsilon\}_{\varepsilon\in[0,\varepsilon_0]}$ of global attractors is continuous at $\varepsilon=0$ if  
$$
\tn{d}_H(\mathcal{A}_\varepsilon,\mathcal{A}_0)=\tn{dist}_H(\mathcal{A}_\varepsilon,\mathcal{A}_0)+\tn{dist}_H(\mathcal{A}_0,\mathcal{A}_\varepsilon)\to 0\quad\tn{as}\quad \varepsilon\to 0,
$$
where  
$$
\tn{dist}_H(A,B)=\sup_{a\in A}\inf_{b\in B}\|a-b\|_{X_\varepsilon^\frac{1}{2}},\quad A,B\subset X_\varepsilon^\frac{1}{2}.
$$

We also recall that the equilibria solutions of \eqref{semilinear_problem} are those which are independent of time, that is, for $\varepsilon \in [0,\varepsilon_0]$, they are the solutions of the elliptic problem  $ A_\varepsilon u^\varepsilon-f(u^\varepsilon)=0$. We denote by $\mathcal{E}_\varepsilon$ the set of the equilibria solutions of $A_\varepsilon$  and we say that $u^\varepsilon_* \in \mathcal{E}_\varepsilon$ is a hyperbolic if
$\sigma(A_\varepsilon-f'(u_*^\varepsilon))\cap \{\mu\in\mathbb{C}\,;\, Re(\mu)=0\}=\varnothing$.

Since hyperbolicity of equilibria is a quite common property for reaction diffusion equations (see \cite{A.N.Carvalho2010}, for example), we assume  $\mathcal{E}_0$ is composed of hyperbolic equilibria only, therefore $\mathcal{E}_0$ is finite and the family $\{\mathcal{E}_\varepsilon\}_{\varepsilon\in [0,\varepsilon_0]}$ is continuous at $\varepsilon= 0$ (see \cite{Carbone2008}); in fact, for $\varepsilon$ sufficiently small, $\mathcal{E}_\varepsilon$ is composed of a finite number of hyperbolic solutions and the semigroups in \eqref{semigroupononlinear} are gradient. Moreover in \cite{Carbone2006} the authors proved that the semigroup $T_0(\cdot)$ is Morse-Smale from which it follows the stability of its phase diagram under suitable perturbations (see \cite{Bortolan}). In \cite{Carbone2006} the authors also proved the gap condition for eigenvalues of the operators $A_\varepsilon$, $\varepsilon\in [0,\varepsilon_0]$, and then the existence of exponential attracting finite dimensional inertial manifolds $\mathcal{M}_\varepsilon$ (of dimension independent of $\epsilon$) containing $\mathcal{A}_\varepsilon$ is ensured. Thus we can restrict the semigroups $T_\varepsilon$ to these inertial manifolds in order to obtain a finite dimensional problem. The robustness these smooth inertial manifolds under regular perturbations is used to ensure the the phase diagram commutativity between attractors. 

Under these assumptions, the authors in \cite{J.M.Arrieta2000a} and \cite{Carbone2008} proved the continuity of attractors (as sets) for a problem similar to  \eqref{semilinear_problem} in the phase space $X_\varepsilon^\frac{1}{2}$, however rate of convergence of attractors was not considered. The aim of this paper is to consider the rate of convergence of attractors as $\varepsilon$ tends to zero.

Some results about rate of convergence of attractors for Morse-Smale problems are founded in \cite{Santamaria2017}, where the authors have obtained an almost optimal rate of convergence of attractors involving the compact convergence of the resolvent operators  
$ \|A_\varepsilon^{-1}-A_0^{-1}\|_{\LL(L^2_{\Omega_0},X_\varepsilon^\frac{1}{2})}$ for a specific thin domain problems. 

Inspired by the works described above, we exhibit a rate of convergence of the attractor of \eqref{semilinear_problem}, in the phase space $X_\varepsilon^\frac{1}{2}$, to the attractor of \eqref{limite_problem}, in the pase space  $X_0^\frac{1}{2}$, in terms of the $\epsilon$ dependence of the diffusion coefficients $p_\varepsilon$. 
All asymptotic objects in our prototype problem (equilibria, unstable manifolds, invariant manifolds and attractors)
converge at this rate which is the rate of convergence of resolvent operators. 

Therefore our main contribution in this paper is to establish the rate of convergence of resolvents operators, eigenvalues and equilibria for the problem \eqref{semilinear_problem} to be able to apply the results of \cite{Santamaria2017} to the continuity of attractors.  In addition, for the particular model considered, our work improves the works \cite{Arrieta,Carbone2008} and \cite{Carbone2006} (in what concerns the continuity of attractors),  where continuity of attractors was proved without any rate. The current paper is the first to consider the rate of convergence of attractors for parabolic problems with localized large diffusion.

This paper is organized as follows. In the Section \ref{Elliptic Problem} we make the study of the elliptic problem in order to find a rate of attraction for the resolvent operators. In the Section \ref{Rate of Convergence of Eigenvalues and Equilibria}  we exhibit a rate of attraction for the eigenvalues and equilibrium points. In Section \ref{Rate of Convergence of Invariant Manifolds} we obtain the rate of convergence of invariant manifolds and in the Section \ref{Rate of Convergence of Attractors} we reduce the system to finite dimensions and we finally obtain a rate of convergence of attractors.   

\bigskip

\noindent{\bf Acknowledgement:} The authors wish to express their gratitude to an anonymous referee who carefully read in detail the whole paper and made suggestions that considerably improved its writing and readability.

\section{Elliptic Problem}\label{Elliptic Problem}
In this section we analyze the solvability of the elliptic problem associated to \eqref{semilinear_problem} in order to obtain the rate of convergence of the resolvent operators. As a consequence we will estimate the convergence of $T_\varepsilon(1)$; that is, the solution operator associated to \eqref{semigroupononlinear} at time one. Later we will transfer such estimate to the convergence of $T_\varepsilon(1)$ restricted to finite dimensional invariant manifold.

Next we prove the convergence of the resolvent operator $A_\varepsilon^{-1}|_{L^2_{\Omega_0}}$ to $A_0^{-1}$ establishing that the rate of this convergence of resolvent operators is $(\|p_\varepsilon-p_0\|_{L^\infty(\Omega_1)}+\varepsilon)^\frac{1}{2}$.  

\begin{lem}\label{Rate_of_convergence}
For $g\in L^2_{\Omega_0}$ with $\|g\|_{L^2}\leq 1$ and $\varepsilon\in [0,\varepsilon_0]$, let $u^\varepsilon$ be the solution of elliptic problem  
\begin{equation}
\begin{cases}
A_\varepsilon u^\varepsilon=g, \quad x\in (0,1),\\
u^\varepsilon_x(0)=u^\varepsilon_x(1)=0.
\end{cases}
\end{equation}
Then there is a constant $C>0$ independent of $\varepsilon$ such that
\begin{equation}
\|u^\varepsilon-u^0\|_{X_\varepsilon^\frac{1}{2}}\leq C(\|p_\varepsilon-p_0\|_{L^\infty(\Omega_1)}+\varepsilon)^\frac{1}{2}.
\end{equation} 
\end{lem}
\begin{proof}
First note that the weak solution $u^\varepsilon$ satisfies
\begin{equation}\label{weak_solution}
\int_0^1 p_\varepsilon u^\varepsilon_x \varphi_x\,dx+\int_0^1 \lambda  u^\varepsilon \varphi\,dx=\int_0^1 g\varphi\,dx,\quad\forall\,\varphi\in X_\varepsilon^{\frac{1}{2}},\,\,\varepsilon\in(0,\varepsilon_0];
\end{equation}
\begin{equation}\label{weak_solution2}
\int_{\Omega_1} p_0 u^0_x \varphi_x\,dx+\int_0^1 \lambda u^0 \varphi\,dx=\int_0^1 g\varphi\,dx,\quad\forall\,\varphi\in X_0^{\frac{1}{2}}.
\end{equation}
If we take $\varphi =u^\varepsilon$ as a test function we get uniform bound for weak solution $u^\varepsilon$ in the spaces $H^1$ and $X_\varepsilon^\frac{1}{2}$ for $\varepsilon\in [0,\varepsilon_0]$. Also the embedding $H^1\subset L^\infty$  gives us an uniform bound for $u^\varepsilon$ in the space $L^\infty$. For estimate $u^\varepsilon_x$ note that
$
-(p_\varepsilon u^\varepsilon_x)_x=g-\lambda u^\varepsilon
$
integrating from $0$ to $x$ we obtain, $p_\varepsilon |u^\varepsilon_x|\leq \|g\|_{L^2}+\lambda\|u^\varepsilon\|_{L^2}$, $x\in [0,1]$ and using \eqref{estimate_diffusion} we have $m_0|u^\varepsilon_x|\leq 1+\lambda$, $x\in [0,1]$, Thus $\|u^\varepsilon_x\|_{L^\infty}\leq m_0^{-1}(1+\lambda)$

In what follows $C$ denotes a positive constant independent of $\varepsilon$ (it may vary from one place no another). 

We define the linear operator $E:X_\varepsilon^\frac{1}{2}\to X_0^\frac{1}{2}$ by 
$$
Eu=\begin{cases} u \quad\tn{in}\quad \Omega_{1\varepsilon}=[0,x_1-\varepsilon]\cup[x_2+\varepsilon,1],  \\
\tn{linear} \quad\tn{in}\quad  \Omega_\varepsilon=[x_1-\varepsilon,x_1]\cup[x_2,x_2+\varepsilon],\\
\bar{u}:=\ds\frac{1}{x_2-x_1}\int_{x_1}^{x_2} u\,dx \quad\tn{in}\quad \Omega_0 , 
\end{cases}
$$
for all $u\in X_\varepsilon^\frac{1}{2}$. If we let $u^\varepsilon-u^0$ as a test function in \eqref{weak_solution} and if we let $E(u^\varepsilon-u^0)$ as a test function in \eqref{weak_solution2}, we have 

\begin{align*}
\int_0^1\! p_\epsilon (u^\varepsilon_x \! -\!u^0_x)^2 dx \!+\!\lambda\! \int_0^1 \!\!(u^\varepsilon\!-\!u^0)^2 dx &\!=-\!\!\int_{\Omega_{1}}\!\!(p_\varepsilon\!-\!p_0)u^0_x (u^\varepsilon_x\!-\!u^0_x) \,dx +\!\! \int_{x_1-\epsilon}^{x_2+\epsilon}\!\!\!g\,(I\!-\!E) (u^\varepsilon\!-\!u^0)\,dx\!\\
&-\!\!\int_{\Omega_\varepsilon}\!\!p_0u^0_x(I\!-\!E)(u^\varepsilon\!-\!u^0)_x \,dx\!-\!\lambda\! \int_{x_1-\varepsilon}^{x_2+\varepsilon}\!\!\! u^0 (I\!-\!E)(u^\varepsilon\!-\!u^0)  \,dx.
\end{align*}


Since $p_\varepsilon$ converges uniformly to $p_0$ in $\Omega_1$, we have $p_\varepsilon$ is uniformly bounded in $\Omega_1$. Thus by H\"{o}der inequality, \eqref{estimate_diffusion} and the uniform bound for weak solution $u^\varepsilon$ and $u^\varepsilon_x$, first term in the right hand side of the above expression can be estimated by $C\|u^\varepsilon-u^0\|_{X_\varepsilon^\frac{1}{2}}(\|p_\varepsilon-p_0\|_{L^\infty(\Omega_1)}+\varepsilon^\frac{1}{2})$. For terms that involve the operator $E$ we have used 
$$
\int_{\Omega_\varepsilon} |E(u^\varepsilon-u^0)_x|\,dx\leq C\|u^\varepsilon-u^0\|_{X_\varepsilon^\frac{1}{2}}\varepsilon^\frac{1}{2} \tn{ and } \|E(u^\varepsilon-u^0)\|_{L^\infty(\Omega_\varepsilon)}\leq C\|u^\varepsilon-u^0\|_{X_\varepsilon^\frac{1}{2}}.
$$
We will prove the first of these inequalities. Denote $v^\varepsilon=u^\varepsilon-u^0$ and assume that $x\in [x_1-\varepsilon,x_1]$. In this case we have
$$
|E(v^\varepsilon)_x|=\Big|\frac{\bar{v^\varepsilon}-v^\varepsilon(x_1-\varepsilon)}{\varepsilon}\Big|\leq \Big|\frac{\bar{v^\varepsilon}-v^\varepsilon(x_1+\varepsilon)}{\varepsilon}\Big|+\Big|\frac{v^\varepsilon(x_1+\varepsilon)-v^\varepsilon(x_1-\varepsilon)}{\varepsilon}\Big|.
$$
Consequently
\begin{align*}
|\bar{v^\varepsilon}-v^\varepsilon(x_1+\varepsilon)|&\leq \frac{1}{x_2-x_1}\int_{x_1}^{x_2}|v^\varepsilon-v^\varepsilon(x_1+\varepsilon)|\,dx \\
& = \frac{1}{x_2-x_1}\Big[\int_{x_1}^{x_1+\varepsilon}+\int_{x_1+\varepsilon}^{x_2-\varepsilon}+\int_{x_2-\varepsilon}^{x_2}|v^\varepsilon-v^\varepsilon(x_1+\varepsilon)|\,dx\Big]
\end{align*}
and we have that
$$
\int_{x_1}^{x_1+\varepsilon}+\int_{x_2-\varepsilon}^{x_2}|v^\varepsilon-v^\varepsilon(x_1+\varepsilon)|\,dx\leq C\|v^\varepsilon\|_{L^\infty}\varepsilon\leq C\|v^\varepsilon\|_{X_\varepsilon^\frac{1}{2}}\varepsilon^\frac{1}{2}.
$$
Now, for $x\in [x_1+\varepsilon,x_2-\varepsilon]$,
$$
|v^\varepsilon(x)-v^\varepsilon(x_1+\varepsilon)|\leq \int_{x_1+\varepsilon}^{x_2-\varepsilon}|v^\varepsilon_x|\,dx\leq \Big(\int_{x_1+\varepsilon}^{x_2-\varepsilon}|v^\varepsilon_x|^2\,dx \Big)^\frac{1}{2},
$$
but
$$
\frac{1}{\varepsilon}\int_{x_1+\varepsilon}^{x_2-\varepsilon}|v^\varepsilon_x|^2\,dx\leq \int_{x_1+\varepsilon}^{x_2-\varepsilon}p_\varepsilon|v^\varepsilon_x|^2\,dx\leq \|v^\varepsilon\|_{X_\varepsilon^\frac{1}{2}}^2
$$
and then 
$$
\int_{x_1+\varepsilon}^{x_2-\varepsilon} |v^\varepsilon(x)-v^\varepsilon(x_1+\varepsilon)|\,dx \leq C \|v^\varepsilon\|_{X_\varepsilon^\frac{1}{2}} \varepsilon^\frac{1}{2}.
$$
We also have
$$
|v^\varepsilon(x_1+\varepsilon)-v^\varepsilon(x_1-\varepsilon)|\leq \int_{x_1-\varepsilon}^{x_1+\varepsilon}|v^\varepsilon_x|\,dx\leq \Big(\int_{x_1-\varepsilon}^{x_1+\varepsilon}|v^\varepsilon_x|^2\,dx \Big)^\frac{1}{2}(2\varepsilon)^\frac{1}{2}\leq C \|v^\varepsilon\|_{X_\varepsilon^\frac{1}{2}} \varepsilon^\frac{1}{2}.
$$
All other estimates are similar and the proof is complete.
\end{proof}

As a consequence of the previous lemma, we have the following result.
\begin{cor}\label{resolvent_estimate}
There is a positive constant $C$ independent of $\varepsilon$ such that
\begin{equation}
\|A_\varepsilon^{-1}-A_0^{-1}\|_{\LL(L^2_{\Omega_0},X_\varepsilon^\frac{1}{2})}\leq C(\|p_\varepsilon-p_0\|_{L^\infty(\Omega_1)}+\varepsilon)^\frac{1}{2}.
\end{equation} 
Furthermore, there is $\varphi\in (\frac{\pi}{2},\pi)$ such that for all $\mu \in\Sigma_{\lambda,\varphi}=\{\mu\in\mathbb{C}:|\tn{arg}(\mu+\lambda)|\leq \varphi \}\setminus\{\mu\in\mathbb{C}:|\mu+\lambda|\leq  r\}$, for small $r>0$, 
\begin{equation}
\|(\mu+A_\varepsilon)^{-1}-(\mu+A_0)^{-1}\|_{\LL(L^2_{\Omega_0},X_\varepsilon^{\frac{1}{2}})}\leq C(\|p_\varepsilon-p_0\|_{L^\infty(\Omega_1)}+\varepsilon)^\frac{1}{2}.
\end{equation}
\end{cor}
\begin{proof}
The first part is an immediate consequence of Lemma \ref{Rate_of_convergence}. Let $\rho(A_\varepsilon)$ be the resolvent set of the operator $A_\varepsilon$, $\varepsilon\in [0,\varepsilon_0]$. If  $\mu\in \rho(-A_\varepsilon)\cap \rho(-A_0)$, we choose $\varphi\in (\frac{\pi}{2},\pi)$ suitable in order to get the sectorial estimates
$$
\|(\mu+A_\varepsilon)^{-1}\|_{\LL(L^2)}\leq \frac{M_\varphi}{|\mu|},\quad\varepsilon\in (0,\varepsilon_0],\quad\tn{and}\quad\|(\mu+A_0)^{-1}\|_{\LL(L^2_{\Omega_0})}\leq \frac{M_\varphi}{|\mu|}.
$$
Therefore,
$\|A_\varepsilon(\mu+A_\varepsilon)^{-1}\|_{\LL(L^2)}\leq 1+M_\varphi$ for $\varepsilon\in (0,\varepsilon_0]$ and $\|A_0(\mu+A_0)^{-1}\|_{\LL(L^2_{\Omega_0})}\leq 1+M_\varphi.$
But if $g\in L^2_{\Omega_0}$, we can write
$$
A_\varepsilon^{\frac{1}{2}}\Big( (\mu+A_\varepsilon)^{-1} - (\mu+A_0)^{-1} \Big)g =  A_\varepsilon(\mu+A_\varepsilon)^{-1} A_\varepsilon^{\frac{1}{2}}(A_\varepsilon^{-1}-A_0^{-1})A_0 (\mu+A_0)^{-1}g
$$
and thus
\begin{align*}
\| (\mu+A_\varepsilon)^{-1} - & (\mu +A_0)^{-1} \|_{\LL(L^2_{\Omega_0},X_\varepsilon^{\frac{1}{2}})} \\
& \leq \|A_\varepsilon(\mu+A_\varepsilon)^{-1}\|_{\LL(L^2_{\Omega_0})}  \|A_\varepsilon^{\frac{1}{2}} (A_\varepsilon^{-1}-A_0^{-1})\|_{\LL(L^2_{\Omega_0})} \| A_0 (\mu+A_0)^{-1}\|_{\LL(L^2_{\Omega_0})}  \\
& \leq \|A_\varepsilon(\mu+A_\varepsilon)^{-1}\|_{\LL(L^2)}  \|A_\varepsilon^{\frac{1}{2}} (A_\varepsilon^{-1}-A_0^{-1})\|_{\LL(L^2_{\Omega_0})} \| A_0 (\mu+A_0)^{-1}\|_{\LL(L^2_{\Omega_0})}  \\
& \leq C (\|p_\varepsilon-p_0\|_{L^{\infty}(\Omega_1)}+\varepsilon)^{\frac{1}{2}},
\end{align*}
for some constant $C=C(\varphi)>0$ independent of $\mu$ and $\varepsilon$. 
\end{proof}

Next we obtain the rate of convergence of nonlinear semigroups. We will follow \cite{Santamaria2017} that improves (for the Morse-Smale case) the results presented by \cite{Arrieta}. For our purposes we just need to consider the time $t=1$.  In fact, for each $\varepsilon$, the time one map $T_\varepsilon(1)$ generates a discrete semigroup $\{T_\varepsilon(1)^n:n\in \N\}$ with the same attractor $\mathcal{A}_\varepsilon$ of $T_\varepsilon(t)$.

\begin{theo}\label{continuity_nonlinear_semigroup}
For each $w_0\in \mathcal{A}_0$, there is a positive constant $C$ independent of $\varepsilon$ such that
$$
\|T_\varepsilon(1)w_0-T_0(1)w_0\|_{X_\varepsilon^{\frac{1}{2}}}\leq C(\|p_\varepsilon-p_0\|_{L^\infty(\Omega_1)}+\varepsilon)^\frac{1}{2}|\log(\|p_\varepsilon-p_0\|_{L^\infty(\Omega_1)}+\varepsilon)|. 
$$
\end{theo}
\begin{proof}
For each $\varepsilon\in [0,\varepsilon_0]$ the operator $A_\varepsilon$ generates an analytic semigroup $\{e^{-A_\varepsilon t}\,;\,t\geq 0\}$ which is given by
$$
e^{-A_\varepsilon t}=\frac{1}{2\pi i}\int_{\Gamma} e^{\mu t}(\mu+A_\varepsilon)^{-1}\,d\mu,
$$
where $\Gamma$ is the boundary of $\Sigma_{\lambda,\varphi}=\{\mu\in \mathbb{C}\,;\, |\tn{arg}(\mu+\lambda)|\leq \varphi\}\setminus \{\mu\in \mathbb{C};|\mu+\lambda|\leq r\}$ for some small $r$ and $\varphi\in (\frac{\pi}{2},\pi)$,  oriented towards the increasing imaginary part.  It is clear that
$$
\|e^{-A_\varepsilon t}\|_{\LL(L^2,X_\varepsilon^{\frac{1}{2}})}\leq M t^{-\frac{1}{2}}e^{-\lambda t},\quad t> 0,\quad \varepsilon\in (0,\varepsilon_0];
$$
$$
\|e^{-A_0 t}\|_{\LL(L^2_{\Omega_0},X_\varepsilon^{\frac{1}{2}})}\leq M t^{-\frac{1}{2}}e^{-\lambda t},\quad t> 0,
$$
where $M$ is a positive constant independent of $\varepsilon$. Thus for $t>0$,
\begin{align*}
\|e^{-A_\varepsilon t}-e^{-A_0 t}\|_{\LL(L^2_{\Omega_0},X_\varepsilon^{\frac{1}{2}})} & \leq \|e^{-A_\varepsilon t}\|_{\LL(L^2_{\Omega_0},X_\varepsilon^{\frac{1}{2}})}+\|e^{-A_0 t}\|_{\LL(L^2_{\Omega_0},X_\varepsilon^{\frac{1}{2}})} \\
& \leq M t^{-\frac{1}{2}}e^{-\lambda t}+M t^{-\frac{1}{2}}e^{-\lambda t}\\
&  \leq 2M t^{-\frac{1}{2}}e^{-\lambda t}.
\end{align*}
Moreover, using the Corollary \ref{resolvent_estimate}, we can prove that
\begin{align*}
\|e^{-A_\varepsilon t}-e^{-A_0 t}\|_{\LL(L^2_{\Omega_0},X_\varepsilon^{\frac{1}{2}})} & \leq \frac{1}{2\pi}\int_\lambda
|e^{\mu t}|\|(\mu+A_\varepsilon)^{-1}-(\mu+A_0)^{-1}\|_{\LL(L^2_{\Omega_0},X_\varepsilon^{\frac{1}{2}})} \, |d\mu| \\
& \leq C(\|p_\varepsilon-p_0\|_{L^{\infty}(\Omega_1)}+\varepsilon)^\frac{1}{2} t^{-1}e^{-\lambda t}.
\end{align*}

Note that the terms $ t^{-\frac{1}{2}}$ and $t^{-1}$ in the the estimates above originates a singularity in the variation of constants formula. This is the main difficulty in estimating the nonlinear semigroups. In \cite{Arrieta} the authors performed an interpolation of these terms together with the rate of convergence of resolvent operators which resulted in the considerable loss in the rate of convergence of attractors. In the situation where the limiting problems is Morse-Smale, the authors in \cite{Santamaria2017} had the same problem, however they used  the following estimate (placed in our context). If we denote $\tau(\varepsilon)=(\|p_\varepsilon-p_0\|_{L^{\infty}(\Omega_1)}+\varepsilon)^\frac{1}{2} $ and $l_\varepsilon(t)=\min\{t^{-\frac{1}{2}},\tau(\varepsilon)t^{-1}\}$, then 
\begin{equation}\label{lemmasantamaria}
\int_{-\infty}^\tau l_\varepsilon(\tau-r)e^{-\lambda(\tau-r)}\,dr\leq C\tau(\varepsilon)|\log(\tau(\varepsilon))|.
\end{equation} 

Since the nonlinear semigroup is given by \eqref{semigroupononlinear}, then for $0<t\leq 1$, we have
\begin{multline*}
\|T_\varepsilon(t)w_0-T_0(t)w_0\|_{X_\varepsilon^{\frac{1}{2}}} \leq \|(e^{-A_\varepsilon t}-e^{-A_0 t})w_0\|_{X_\varepsilon^{\frac{1}{2}}} \\ +\int_0^{t} \|e^{-A_\varepsilon(t-s)}f(T_\varepsilon(s)w_0)-e^{-A_0(t-s)}f(T_0(s)w_0)\|_{X_\varepsilon^{\frac{1}{2}}}\,ds,
\end{multline*}
but
\begin{align*}
\int_0^{t} \|e^{-A_\varepsilon(t-s)}f(T_\varepsilon(s)w_0)-& e^{-A_0(t-s)}f(T_0(s)w_0)\|_{X_\varepsilon^{\frac{1}{2}}}\,ds \\
&\leq \int_0^{t} \|e^{-A_\varepsilon(t-s)}  [f(T_\varepsilon(s)w_0) - f(T_0(s)w_0) ]\|_{X_\varepsilon^{\frac{1}{2}}}\,ds \\
& + \int_0^{t} \| [e^{-A_\varepsilon(t-s)}-e^{-A_0(t-s)}]f(T_0(s)w_0) \|_{X_\varepsilon^{\frac{1}{2}}}\,ds\\ 
&\leq C \int_0^{t} (t-s)^{-\frac{1}{2}} \|T_\varepsilon(s)w_0 - T_0(s)w_0]\|_{X_\varepsilon^{\frac{1}{2}}}\,ds \\
& + C\int_0^{t} l_\varepsilon(t-s) e^{-\lambda(t-s)}\|f(T_0(s)w_0) \|_{X_\varepsilon^{\frac{1}{2}}} \,ds.
\end{align*}
By etimate \eqref{lemmasantamaria} and  Gronwall's inequality (see \cite{A.N.Carvalho2010} and  \cite{Henry1980}  for singular Gronwall Lemma) we have
$$
\|T_\varepsilon(t)w_0-T_0(t)w_0\|_{X_\varepsilon^{\frac{1}{2}}}\leq C\tau(\varepsilon)|\log(\tau(\varepsilon))|+C\tau(\varepsilon)|\log(\tau(\varepsilon))|e^{Kt},
$$
where $K=C\Gamma(\frac{1}{2})^{\frac{1}{2}}$. Now the result follows taking $t=1$.
\end{proof}

\section{Rate of Convergence of Eigenvalues and Equilibria}\label{Rate of Convergence of Eigenvalues and Equilibria}
In this section we will obtain the rate of convergence of eigenvalues, spectral projection and equilibrium points.

The convergence of eigenvalues and eigenfunctions of the linear operators was proved in \cite{Rodriguez-Bernal1998} and the properties about the compact convergence of the spectral projections was studied in details in  several works (see, for example, \cite{Carbone2008,Carvalho2006} and \cite{Rodriguez-Bernal1998}). 

In the next result we will follow \cite{Arrieta}, where it was considered rate of convergence for the  eigenvalues and spectral projections. We also state the gap condition proved in \cite{Carbone2006}.     

\begin{lem}\label{lemma-evconv+gap}
If $\lambda^\varepsilon\in \sigma(A_\varepsilon)$, $\varepsilon\in [0,\varepsilon_0]$, and $\lambda^\varepsilon\overset{\varepsilon\to 0}\longrightarrow \lambda^0$, then 
$$
|\lambda^\varepsilon-\lambda^0|\leq C(\|p_\varepsilon-p_0\|_{L^\infty(\Omega_1)}+\varepsilon)^\frac{1}{2}.
$$ 
Moreover, if we denote $\sigma(A_\varepsilon)=\{\lambda_i^\varepsilon\}_{i=0}^\infty$ (ordered and counting multiplicity), we have the following gap condition 
$$
\lambda_{i+1}^\varepsilon-\lambda_{i}^\varepsilon \overset{i\to\infty}{\longrightarrow} \infty.
$$
\end{lem}
\begin{proof}
Let $\lambda^0 \in \sigma(A_0)$ be an isolated eigenvalue. We consider an appropriated closed curve $\Gamma$ in $\rho(-A_0)$ around $\lambda^0$ and define the spectral projection
$$
Q_\varepsilon=\frac{1}{2\pi i}\int_{\Gamma} (\mu+A_\varepsilon)^{-1}\,d\mu,\quad \varepsilon\in[0,\varepsilon_0].
$$
It follows from Corollary \ref{resolvent_estimate} that
$$
\|Q_\varepsilon-Q_0\|_{\LL(L^2_{\Omega_0},X_\varepsilon^\frac{1}{2})}\leq C(\|p_\varepsilon-p_0\|_{L^\infty(\Omega_1)}+\varepsilon)^\frac{1}{2}.
$$
If we have $\lambda^\varepsilon \in \sigma(A_\varepsilon)$ such that $\lambda^\varepsilon \overset{\varepsilon\to 0}\longrightarrow \lambda^0$ then for $\varepsilon$ sufficiently small there is $u^0\in\tn{Ker}(\lambda_0-A_0)$ with $\|u^0\|_{X_\varepsilon^\frac{1}{2}}=1$ such that $Q_\varepsilon u^0$ is eigenfunction of $A_\varepsilon$ associated with $\lambda^\varepsilon$, thus
$$
|\lambda^\varepsilon-\lambda^0|\leq \|(\lambda^\varepsilon-\lambda^0)u^0\|_{X_\varepsilon^\frac{1}{2}}=\|\lambda^\varepsilon Q_0u^0-\lambda^0u^0\|_{X_\varepsilon^\frac{1}{2}},
$$
but
\begin{align*}
\|\lambda^\varepsilon Q_0u^0-\lambda^0u^0\|_{X_\varepsilon^\frac{1}{2}} &\leq \|\lambda^\varepsilon Q_0u^0-\lambda^0Q_\varepsilon u^0+\lambda^0Q_\varepsilon u^0 -\lambda^0u^0 \|_{X_\varepsilon^\frac{1}{2}}\\
& \leq |\lambda^\varepsilon\lambda^0|\|\frac{1}{\lambda^0} Q_0u^0-\frac{1}{\lambda^\varepsilon}Q_\varepsilon u^0\|_{X_\varepsilon^\frac{1}{2}}+\|\lambda^0(Q_\varepsilon-Q_0)u^0\|_{X_\varepsilon^\frac{1}{2}}\\
& \leq |\lambda^\varepsilon\lambda^0|\|A_0^{-1}Q_0u^0-A_\varepsilon^{-1} Q_\varepsilon u^0\|_{X_\varepsilon^\frac{1}{2}}+|\lambda^0|\|(Q_\varepsilon-Q_0)u^0\|_{X_\varepsilon^\frac{1}{2}},
\end{align*}
and
\begin{align*}
\|A_0^{-1}Q_0u^0-A_\varepsilon^{-1} Q_\varepsilon u^0\|_{X_\varepsilon^\frac{1}{2}}& = \|A_0^{-1}Q_0u^0-A_\varepsilon^{-1}Q_0u^0+A_\varepsilon^{-1}Q_0u^0 -A_\varepsilon^{-1} Q_\varepsilon u^0\|_{X_\varepsilon^\frac{1}{2}}\\
&\leq \|(A_0^{-1}-A_\varepsilon^{-1})Q_0u^0\|_{X_\varepsilon^\frac{1}{2}}+\|A_\varepsilon^{-1}(Q_0-Q_\varepsilon)u^0\|_{X_\varepsilon^\frac{1}{2}}.
\end{align*}
Hence the estimate for $|\lambda^\varepsilon-\lambda^0|$ follows. 

In \cite{Carbone2006} it was proved that the eigenvalues of $-A_\varepsilon$ has a gap condition by characterizing these eigenvalues, that is, if $\sigma(-A_\varepsilon)=\{\mu_i^\varepsilon\}_{i=0}^\infty$ then
$$
\mu_i^\varepsilon=-\frac{1}{l^2}i^2\pi^2+o(i)
$$      
as $i\to \infty$, where $l=\int_0^1p_\varepsilon(s)^{-\frac{1}{2}}\,ds$. Consequently  $\lambda_i^\varepsilon-\lambda_{i+1}^\varepsilon\overset{i\to\infty}\longrightarrow \infty.$
\end{proof}

Recall that we denote $\mathcal{E}_\varepsilon$ the set of the equilibria solutions of the $A_\varepsilon$ and we assume that $\mathcal{E}_0$ is composed of hyperbolic equilibria, thus for $\varepsilon$ sufficiently small $\mathcal{E}_\varepsilon$ is composed of finite number of hyperbolic equilibria. The rate of convergence of equilibrium points can be obtained as follows.

\begin{theo}
Let $u_*^0\in\mathcal{E}_0$. Then for $\varepsilon$ sufficiently small (we still denote $\varepsilon\in (0,\varepsilon_0]$), there is $\delta>0$ such that the equation $A_\varepsilon u-f(u)=0$ has only solution $u_*^\varepsilon\in \{u\in X_\varepsilon^\frac{1}{2}\,;\,\|u-u_*^0\|_{X_\varepsilon^\frac{1}{2}}\leq \delta\}$. Moreover
\begin{equation}\label{rate_of_equilibrium}
\|u_*^\varepsilon-u_*^0\|_{X_\varepsilon^\frac{1}{2}}\leq C(\|p_\varepsilon-p_0\|_{L^\infty(\Omega_1)}+\varepsilon)^\frac{1}{2}.
\end{equation} 
\end{theo}
\begin{proof}
The proof is the same as given in \cite{Carbone2008}. Here we just need to prove the estimates \eqref{rate_of_equilibrium}. We have $u_*^\varepsilon$ and $u_*^0$ given by 
$$
u_*^0=(A_0+V_0)^{-1}[f(u_*^0)+V_0u_*^0]\quad\tn{and}\quad u_*^\varepsilon=(A_\varepsilon+V_0)^{-1}[f(u_*^\varepsilon)+V_0 u_*^\varepsilon],
$$ 
where $V_0=-f'(u_*^0)$. Thus
\begin{align*}
\|u_*^\varepsilon-u_*^0\|_{X_\varepsilon^\frac{1}{2}}&\leq \|(A_\varepsilon+V_0)^{-1}[f(u_*^\varepsilon)+V_0 u_*^\varepsilon]-(A_0+V_0)^{-1}[f(u_*^0)+V_0u_*^0]\|_{X_\varepsilon^\frac{1}{2}}\\
&\leq \|(A_\varepsilon+V_0)^{-1}[f(u_*^\varepsilon)-f(u_*^0)+V_0(u_*^\varepsilon-u_*^0)]\|_{X_\varepsilon^\frac{1}{2}}\\
&+ \|[(A_\varepsilon+V_0)^{-1}-(A_0+V_0)^{-1}][f(u_*^0)+V_0 u_*^0]\|_{X_\varepsilon^\frac{1}{2}}.
\end{align*}
We can prove that
$$
(A_\varepsilon+V_0)^{-1}-(A_0+V_0)^{-1}=[I-(A_\varepsilon+V_0)^{-1}V_0](A_\varepsilon^{-1}-A_0^{-1})[I-V_0(A_0+V_0)^{-1}],
$$
which implies
$
\|[(A_\varepsilon+V_0)^{-1}-(A_0+V_0)^{-1}][f(u_*^0)+V_0 u_*^0]\|_{X_\varepsilon^\frac{1}{2}}\leq C(\|p_\varepsilon-p_0\|_{L^\infty(\Omega_1)}+\varepsilon)^\frac{1}{2}.
$

Now we denote $z^\varepsilon=f(u_*^\varepsilon)-f(u_*^0)+V_0(u_*^\varepsilon-u_*^0)$. Since $f$ is continuously differentiable, for all $\delta>0$ there is $\varepsilon$ sufficiently small such that 
$
\|z^\varepsilon\|_{X_\varepsilon^\frac{1}{2}}\leq \delta\|u_*^\varepsilon-u_*^0\|_{X_\varepsilon^\frac{1}{2}},
$
thus
$$
\|(A_\varepsilon+V_0)^{-1}z^\varepsilon\|_{X_\varepsilon^\frac{1}{2}}\leq \delta \|(A_\varepsilon+V_0)^{-1}\|_{\LL(L^2_{\Omega_0},X_\varepsilon^\frac{1}{2})}\|u_*^\varepsilon-u_*^0\|_{X_\varepsilon^\frac{1}{2}}. 
$$
We choose $\delta$ sufficiently small such that $\delta\|(A_\varepsilon+V_0)^{-1}\|_{\LL(L^2_{\Omega_0},X_\varepsilon^\frac{1}{2})}\leq\frac{1}{2}$, thus
$$
\|u_*^\varepsilon-u_*^0\|_{X_\varepsilon^\frac{1}{2}}\leq C(\|p_\varepsilon-p_0\|_{L^\infty(\Omega_1)}+\varepsilon)^\frac{1}{2}+\frac{1}{2}\|u_*^\varepsilon-u_*^0\|_{X_\varepsilon^\frac{1}{2}}.
$$
\end{proof}
\begin{cor}
The family  $\{\mathcal{E}_\varepsilon\}_{\varepsilon\in (0,\varepsilon_0]}$ is continuous at $\varepsilon=0$. Moreover if $\mathcal{E}_0=\{u_*^{0,1},...,u_*^{0,k}\}$ then for $\varepsilon$ sufficiently small, $\mathcal{E}_\varepsilon=\{u_*^{\varepsilon,1},...,u_*^{\varepsilon,k}\}$  and 
$$
\|u_*^{\varepsilon,i}-u_*^{0,i}\|_{X_\varepsilon^{\frac{1}{2}}}\leq C(\|p_\varepsilon-p_0\|_{L^\infty(\Omega_1)}+\varepsilon)^\frac{1}{2},\quad i=1,...,k.
$$
\end{cor}

\begin{rem}
The rate of convergence of eigenvalues, eigenfunctions and equilibrium points is better than rate of convergence of the nonlinear semigroup. This fact is related to the estimates that we obtained for the linear semigroup in the fractional power space and, we will see that the rate of convergence of attractors of problems \eqref{semilinear_problem} has a loss with respect to the rate of resolvent operators.   A class of problems where the rate of convergence of the resolvent operators is the same rate of convergence of the attractors is presented in \cite{CarvalhoPires2017}.
\end{rem}

\section{Rate of Convergence of Invariant Manifolds}\label{Rate of Convergence of Invariant Manifolds}
In this section we characterize the invariant manifolds $\mathcal{M}_\varepsilon$ locally as a graph of a Lipschitz function, and we guarantee that $\mathcal{M}_\varepsilon$ approaches to the invariant manifold $\mathcal{M}_0$ when the parameter $\varepsilon$ goes to zero.  This result  will be fundamental to reduce the study of the asymptotic dynamics of the problem \eqref{semilinear_problem} to a finite dimension.

The spectrum of $-A_\varepsilon$, $\varepsilon\in[0,\varepsilon_0]$, ordered and counting multiplicity is given by 
$$
...-\lambda^\varepsilon_m<-\lambda^\varepsilon_{m-1}<...<-\lambda_0^\varepsilon<0
$$
with $\{\varphi_i^\varepsilon\}_{i=0}^\infty$ the eigenfunctions related. We consider the spectral projection onto the space generated by the first $m$ eigenvalues, that is, if $\Gamma$ is an appropriated closed curve in $\rho(-A_0)$ around $\{-\lambda_0^0,...,-\lambda_{m-1}^0 \}$, then
$$
Q_\varepsilon=\frac{1}{2\pi i}\int_{\Gamma} (\mu+A_\varepsilon)^{-1}\,d\mu,\quad \varepsilon\in[0,\varepsilon_0].
$$
We observe that $Q_\varepsilon$ is a projection of finite rank and then there is an isomorphism from $Q_\varepsilon X_\varepsilon^\frac{1}{2}=\tn{span}[\varphi_0^\varepsilon,...,\varphi_{m-1}^\varepsilon]$ onto $\R^m$. Thus we can decompose $X_\varepsilon^\frac{1}{2}=Y_\varepsilon\oplus Z_\varepsilon$, where $Y_\varepsilon=Q_\varepsilon X_\varepsilon^\frac{1}{2}$ and $Z_\varepsilon=(I-Q_\varepsilon)X_\varepsilon^\frac{1}{2}$ and we define $A_\varepsilon^+=A_\varepsilon|_{Y_\varepsilon}$ and  $A_\varepsilon^-=A_\varepsilon|_{Z_\varepsilon}$ $\varepsilon\in [0,\varepsilon_0].$  The following estimates are valid (see \cite{CarvalhoPires2017}). 
\begin{itemize}
\item[(i)]$\|e^{-A_\varepsilon^+  t}z\|_{X_\varepsilon^\frac{1}{2}}\leq M e^{\beta t}\|z\|_{X_\varepsilon^\frac{1}{2}},\quad t\leq 0,\quad z\in Y_\varepsilon,$

\item[(ii)]$\|e^{-A^-_\varepsilon  t}z\|_{X_\varepsilon^\frac{1}{2}}\leq Me^{-\gamma t}\|z\|_{X_\varepsilon^\frac{1}{2}} ,\quad t\geq 0,\quad z\in Z_\varepsilon,$
\item[(iii)]$\|e^{-A_\varepsilon^+ t}-e^{-A_0^+ t}\|_{\LL(L^2_{\Omega_0},X_\varepsilon^\frac{1}{2})}\leq Me^{\beta t}l_\varepsilon(-t)
, \quad t\leq 0,$  
\item[(iv)] $\|e^{-A_\varepsilon^- t}-e^{-A_0^- t}\|_{\LL(L^2_{\Omega_0},X_\varepsilon^\frac{1}{2})}\leq Me^{-\gamma t}l_\varepsilon(t), \quad t> 0.$
\end{itemize}
where $l_\varepsilon(t)=\min\{t^{-\frac{1}{2}},(\|p_\varepsilon-p_0\|_{L^{\infty}(\Omega_1)}+\varepsilon)^\frac{1}{2} t^{-1}\}$, $\gamma=\lambda_{m}^0-\alpha(\lambda_{m}^0)^\frac12$, $-\beta = \lambda_{m-1}^0 +\alpha(\lambda_{m-1}^0)^\frac12$ for $\alpha>0$ small and $M$ independent of $\varepsilon$ and $m$.

As a consequence of Lemma \ref{lemma-evconv+gap}, for suitably chosen $\alpha$, the sum of $\gamma$ and $\beta$, as a function of $m$, can be made as large as we wish. The choice of $m$ will be used as a parameter in the proof of existence of the invariant manifold.

\begin{theo} For sufficiently large $m$ and $\varepsilon$ small, there is an invariant manifold $\mathcal{M}_\varepsilon$ for \eqref{semilinear_problem} given by
$$
\mathcal{M}_\varepsilon=\{u^\varepsilon\in X_\varepsilon^\frac{1}{2}\,;\, u^\varepsilon = Q_\varepsilon u^\varepsilon+s_{*}^\varepsilon(Q_\varepsilon u^\varepsilon)\},\quad \varepsilon\in[0,\varepsilon_0],
$$ 
where $s_\ast^\varepsilon:Y_\varepsilon\to Z_\varepsilon$ is a Lipschitz continuous map satisfying
\begin{equation}\label{estimate_invariant_manifold}
|\!|\!|s_\ast^\varepsilon-s_\ast^0 |\!|\!|=\sup_{v\in Y_0}\|s_\ast^\varepsilon(v)-s_\ast^0(v)\|_{X_\varepsilon^\frac{1}{2}}\leq C\tau(\varepsilon)|\log(\tau(\varepsilon))|,
\end{equation}
where $\tau(\varepsilon)=(\|p_\varepsilon-p_0\|_{L^\infty(\Omega_1)}+\varepsilon)^\frac{1}{2}$ and $C$  is a constant independent of $\varepsilon$.
The invariant manifold $\mathcal{M}_\varepsilon$ is exponentially attracting and the global attractor $\mathcal{A}_\varepsilon$ of the problem \eqref{semilinear_problem} lies  in $\mathcal{M}_\varepsilon$. The flow on $\mathcal{A}_\varepsilon$ is given by
$$
u^\varepsilon(t)=v^\varepsilon(t)+s_\ast^\varepsilon(v^\varepsilon(t)), \quad t\in\R,
$$ 
where $v^\varepsilon(t)$ satisfies
$$
\dot{v^\varepsilon}+A_\varepsilon^+v^\varepsilon=Q_\varepsilon f(v^\varepsilon+ s_\ast^\varepsilon(v^\varepsilon(t))).
$$
\end{theo}
\begin{proof} The  proof is well known in the theory of invariant manifolds (see \cite{A.N.Carvalho2010} chapter 8). Here we just need to prove the estimate \eqref{estimate_invariant_manifold}.

For given $D>0$ and $0<\Delta$ we consider the set
$$
\Sigma_\varepsilon = \Big\{s^\varepsilon: Y_\varepsilon\to Z_\varepsilon\,;\, |\!|\!| s^\varepsilon|\!|\!|\leq D  \tn{ and }\|s^\varepsilon(v)-s^\varepsilon(\tilde{v})\|_{X_\varepsilon^{\frac{1}{2}}}\leq \Delta \|v-\tilde{v}\|_{X_\varepsilon^{\frac{1}{2}}}\Big\}.
$$
It is not difficult to see that $(\Sigma_\varepsilon, |\!|\!| \cdot |\!|\!|)$ is a complete metric space with the uniform convergence topology. We write the solution $u^\varepsilon$ of \eqref{semilinear_problem} as $u^\varepsilon=v^\varepsilon+z^\varepsilon$, with $v^\varepsilon\in Y_\varepsilon$ and $z^\varepsilon\in Z_\varepsilon$ and since $Q_\varepsilon$ and $I-Q_\varepsilon$ commute with $A_\varepsilon$,  we can write 
\begin{equation}\label{couple_system}
\begin{cases}
v_t^\varepsilon+A_\varepsilon^+ v^\varepsilon = Q_\varepsilon f(v^\varepsilon+z^\varepsilon):=H_\varepsilon(v^\varepsilon,z^\varepsilon)\\
z_t^\varepsilon+A_\varepsilon^- z^\varepsilon = (I-Q_\varepsilon)f(v^\varepsilon+z^\varepsilon):=G_\varepsilon(v^\varepsilon,z^\varepsilon).
\end{cases}
\end{equation}
By assumption there is a certain $\rho>0$ such that for all $ v^\varepsilon,\tilde{v}^\varepsilon\in Y_\varepsilon$ and $z^\varepsilon,\tilde{z}^\varepsilon\in Z_\varepsilon$,
\begin{itemize} 
\item[]$\ds\|H_\varepsilon( v^\varepsilon,z^\varepsilon)\|_{X_\varepsilon^\frac{1}{2}}\leq \rho,$ $\|G_\varepsilon( v^\varepsilon,z^\varepsilon)\|_{X_\varepsilon^\frac{1}{2}}\leq \rho,$
\item[]$\|H_\varepsilon( v^\varepsilon,z^\varepsilon)-H_\varepsilon( \tilde{v}^\varepsilon,\tilde{z}^\varepsilon)\|_{X_\varepsilon^\frac{1}{2}}\leq \rho (\|v^\varepsilon-\tilde{v}_\varepsilon\|_{X_\varepsilon^\frac{1}{2}} +\|z^\varepsilon-\tilde{z}_\varepsilon \|_{X_\varepsilon^\frac{1}{2}} ),$
\item[]$
\|G_\varepsilon( v^\varepsilon,z^\varepsilon)-G_\varepsilon( \tilde{v}^\varepsilon,\tilde{z}^\varepsilon)\|_{X_\varepsilon^\frac{1}{2}}\leq \rho (\|v^\varepsilon-\tilde{v}_\varepsilon\|_{X_\varepsilon^\frac{1}{2}} +\|z^\varepsilon-\tilde{z}_\varepsilon \|_{X_\varepsilon^\frac{1}{2}} ).$
\end{itemize}
Also, choosing $m$ suitably large and then $\varepsilon$ small we have that
 \begin{itemize}
\item[] $\ds\rho M\gamma^{-1}\leq D$, $0 < \gamma+\beta-\rho M(1+\Delta)$, $\ds\frac{\rho M^2(1+\Delta)}{\gamma+\beta-\rho M(1+\Delta)}\leq \Delta$,
\item[] $\ds\rho M \gamma^{-1}+\frac{\rho^2M^2(1+\Delta)\beta^{-1}}{\gamma+\beta-\rho M(1+\Delta)} \leq \frac{1}{2}$, $\ds L=\Big[\rho M+\frac{\rho^2 M^2(1+\Delta)(1+M)}{\gamma+\beta-\rho M(1+\Delta)} \Big]$, $\gamma-L> 0$.
\end{itemize}

\smallskip

We will divide the proof in three parts.

\textbf{Part 1}(Existence) Let $s^\varepsilon\in \Sigma_\varepsilon$ and $v^\varepsilon(t)=v^\varepsilon(t,\tau,\eta,s^\varepsilon)$ be the solution of 
$$
\begin{cases}
v_t^\varepsilon+A_\varepsilon^+ v^\varepsilon=H_\varepsilon(v^\varepsilon,s^\varepsilon(v^\varepsilon)),\quad  t<\tau \\
v^\varepsilon(\tau)=\eta.
 \end{cases}
$$
We define $\Phi_\varepsilon: \Sigma_\varepsilon\to\Sigma_\varepsilon$ by
$$
\Phi_\varepsilon(s^\varepsilon)(\eta)=\int_{-\infty}^{\tau} e^{-A_\varepsilon^- (\tau - r)}G_\varepsilon(v^\varepsilon(r),s^\varepsilon(v^\varepsilon(r)) )\,dr.
$$
Then 
$$ 
\|\Phi_\varepsilon(s^\varepsilon)(\eta)\|_{X_\varepsilon^{\frac{1}{2}}}\leq \rho M\int_{-\infty}^\tau  e^{-\gamma(\tau - r)} \,dr =\rho M\gamma^{-1}\leq D.
$$
For $s^\varepsilon, \tilde{s^\varepsilon}\in \Sigma_\varepsilon$, $\eta,\tilde{\eta}\in Y_\varepsilon$,  $v^\varepsilon(t)=v^\varepsilon(t,\tau,\eta,s^\varepsilon)$ and $\tilde{v}^\varepsilon(t)=\tilde{v}^\varepsilon(t,\tau,\tilde{\eta},\tilde{s}^\varepsilon)$ we have
\begin{multline*}
v^\varepsilon(t)-\tilde{v}^\varepsilon(t) = e^{-A_\varepsilon^+ (t-\tau)}(\eta-\tilde{\eta}) \\ +\int_{\tau}^t   e^{-A_\varepsilon^+ (t - r)}[H_\varepsilon(v^\varepsilon(r),s^\varepsilon(v^\varepsilon(r)) )-H_\varepsilon(\tilde{v}^\varepsilon(r),\tilde{s}^\varepsilon(\tilde{v}^\varepsilon(r)) )] \,dr.
\end{multline*}
Thus,
\begin{align*}
\|v^\varepsilon(t) & -\tilde{v}^\varepsilon(t)\|_{X_\varepsilon^{\frac{1}{2}}}  \leq Me^{\beta(t-\tau)}\|\eta-\tilde{\eta}\|_{X_\varepsilon^{\frac{1}{2}}} \\
 & + M \int_t^{\tau}  e^{\beta(t - r)}\|H_\varepsilon(v^\varepsilon(r),s^\varepsilon(v^\varepsilon(r)) )-H_\varepsilon(\tilde{v}^\varepsilon(r),\tilde{s}^\varepsilon(\tilde{v}^\varepsilon(r)) ) \|_{X_\varepsilon^{\frac{1}{2}}} \,dr \\
& \leq Me^{\beta(t-\tau)}\|\eta-\tilde{\eta}\|_{X_\varepsilon^{\frac{1}{2}}} \\
 & + \rho M \int_t^{\tau}  e^{\beta(t-r)}[\|v^\varepsilon(r)-\tilde{v}^\varepsilon(r) \|_{X_\varepsilon^{\frac{1}{2}}}+\| s^\varepsilon(v^\varepsilon(r)) -\tilde{s}^\varepsilon(\tilde{v}^\varepsilon(r)) \|_{X_\varepsilon^{\frac{1}{2}}}] \,dr\\
& \leq Me^{\beta(t-\tau)}\|\eta-\tilde{\eta}\|_{X_\varepsilon^{\frac{1}{2}}} \\
 & + \rho M \int_t^{\tau}  e^{\beta(t-r)}[(1+\Delta)\|v^\varepsilon(r)-\tilde{v}^\varepsilon(r) \|_{X_\varepsilon^{\frac{1}{2}}}+\|s^\varepsilon(\tilde{v}^\varepsilon(r)) -\tilde{s}^\varepsilon(\tilde{v}^\varepsilon(r)) \|_{X_\varepsilon^{\frac{1}{2}}}] \,dr \\
& \leq Me^{\beta(t-\tau)}\|\eta-\tilde{\eta}\|_{X_\varepsilon^{\frac{1}{2}}} \\ 
&+ \rho M(1+\Delta) \int_t^{\tau}  e^{\beta(t - r)}\|v^\varepsilon(r)-\tilde{v}^\varepsilon(r) \|_{X_\varepsilon^{\frac{1}{2}}} \,dr+ \rho M |\!|\!| s^\varepsilon-\tilde{s}^\varepsilon |\!|\!|\int_t^{\tau}  e^{\beta(t-r)} \,dr\\
& \leq Me^{\beta(t-\tau)}\|\eta-\tilde{\eta}\|_{X_\varepsilon^{\frac{1}{2}}} \\ 
&+ \rho M(1+\Delta) \int_t^{\tau}  e^{\beta(t - r)}\|v^\varepsilon(r)-\tilde{v}^\varepsilon(r) \|_{X_\varepsilon^{\frac{1}{2}}} \,dr+ \rho M\beta^{-1} |\!|\!| s^\varepsilon-\tilde{s}^\varepsilon |\!|\!| e^{\beta(t-\tau)}.
\end{align*} 
By Gronwall's inequality, 
$$
\|v^\varepsilon(t)-\tilde{v}^\varepsilon(t)\|_{X_\varepsilon^{\frac{1}{2}}}\leq \Big[M\|\eta-\tilde{\eta}\|_{X_\varepsilon^\frac{1}{2}}+\rho M \beta^{-1}|\!|\!| s^\varepsilon-\tilde{s}^\varepsilon |\!|\!|  \Big]  e^{[\rho M(1+\Delta)-\beta](\tau-t)}.
$$ 
Thus
\begin{align*}
\|\Phi_\varepsilon(s^\varepsilon)(\eta) -\Phi_\varepsilon(\tilde{s}^\varepsilon)(\tilde{\eta})\|_{X_\varepsilon^{\frac{1}{2}}}& \leq \int_{-\infty}^{\tau} \| e^{-A_\varepsilon^- (\tau - r)}[G_\varepsilon(v^\varepsilon(r),s^\varepsilon(v^\varepsilon(r)))-G_\varepsilon(\tilde{v}^\varepsilon(r),\tilde{s}^\varepsilon(\tilde{v}^\varepsilon(r)))] \|_{X_\varepsilon^\frac{1}{2}}\,dr \\
&\leq \frac{\rho M^2(1+\Delta)}{\gamma+\beta-\rho M(1+\Delta)}\|\eta-\tilde{\eta}\|_{X_\varepsilon^\frac{1}{2}}+\frac{\rho^2M^2(1+\Delta)\beta^{-1}}{\gamma+\beta-\rho M(1+\Delta)}|\!|\!| s^\varepsilon-\tilde{s}^\varepsilon |\!|\!|\\
&\leq \Delta\|\eta-\tilde{\eta}\|_{X_\varepsilon^\frac{1}{2}}+\frac12|\!|\!| s^\varepsilon-\tilde{s}^\varepsilon |\!|\!|.
\end{align*}
Therefore, making $s^\varepsilon={\tilde s}^\varepsilon$ we have that $\Phi_\varepsilon$ takes $\Sigma_\varepsilon$ into itself and, making $\eta=\tilde\eta$ and taking the suppremum in $\eta$ we have that $\Phi_\varepsilon$ is a contraction on $\Sigma_\varepsilon$, hence there is a unique  $s_\ast^\varepsilon \in \Sigma_\varepsilon$ which is a fixed point of $\Phi_\varepsilon$.

Now, let $(\bar{v}^\varepsilon,\bar{z}^\varepsilon)\in \mathcal{M}_\varepsilon$, $\bar{z}^\varepsilon=s_\ast^\varepsilon(\bar{v}^\varepsilon)$ and let $v_{s_\ast}^\varepsilon(t)$ be the solution of 
$$
\begin{cases}
v_t^\varepsilon+A_\varepsilon^+ v^\varepsilon=H_\varepsilon(v^\varepsilon,s_*^\varepsilon(v^\varepsilon)),\quad  t<\tau \\
v^\varepsilon(0)=\bar{v}^\varepsilon.
 \end{cases}
$$ 
Thus, $\{(v_{s_*}^\varepsilon(t), s_*^\varepsilon(v_{s_*}^\varepsilon(t))\}_{t\in\R}$ defines a curve on $\mathcal{M}_\varepsilon$. But the only solution of equation 
$$
z_t^\varepsilon+A_\varepsilon^- z^\varepsilon=G_\varepsilon(v_{s_*}^\varepsilon(t),s_*^\varepsilon(v_{s_*}^\varepsilon(t)))
$$ 
which stays bounded when $t\to-\infty$ is given by
$$
z_{s_*}^\varepsilon=\int_{-\infty}^t e^{-A_\varepsilon^- (t-r)}G_\varepsilon(v_{s_*}^\varepsilon(t),s_*^\varepsilon(v_{s_*}^\varepsilon(t)))\,dr = s_*^\varepsilon(v_{s_*}^\varepsilon(t)).
$$
Therefore $(v_{s_*}^\varepsilon(t), s_*^\varepsilon(v_{s_*}^\varepsilon(t))$ is a solution of \eqref{couple_system} through $(\bar{v}^\varepsilon,\bar{z}^\varepsilon)$ and thus $\mathcal{M}_\varepsilon$ is a invariant manifold for \eqref{semilinear_problem}.

\textbf{Part 2}(Estimate) Now we will prove the estimate \eqref{estimate_invariant_manifold}. For $\eta \in Y_0$,  we have
\begin{align*}
\|s^\varepsilon_*(\eta)-s^0_*(\eta)\|_{X_\varepsilon^\frac{1}{2}}& \leq \int_{-\infty}^\tau \|e^{-A_\varepsilon^-(\tau-r)}G_\varepsilon(v^\varepsilon,s_*^\varepsilon(v^\varepsilon))-e^{-A_0^-(\tau-r)}G_0(v^0,s_*^0(v^0))\|_{X_\varepsilon^\frac{1}{2}}\,dr\\
& \leq \int_{-\infty}^\tau \|e^{-A_\varepsilon^-(\tau-r)}G_\varepsilon(v^\varepsilon,s_*^\varepsilon(v^\varepsilon))-e^{-A_\varepsilon^-(\tau-r)}G_\varepsilon(v^0,s_*^0(v^0))\|_{X_\varepsilon^\frac{1}{2}}\,dr\\
& + \int_{-\infty}^\tau \|e^{-A_\varepsilon^-(\tau-r)}G_\varepsilon(v^0,s_*^0(v^0))-e^{-A_\varepsilon^-(\tau-r)}G_0(v^0,s_*^0(v^0))\|_{X_\varepsilon^\frac{1}{2}}\,dr\\
& + \int_{-\infty}^\tau \|e^{-A_\varepsilon^-(\tau-r)}G_0(v^0,s_*^0(v^0))-e^{-A_0^-(\tau-r)}G_0(v^0,s_*^0(v^0))\|_{X_\varepsilon^\frac{1}{2}}\,dr.
\end{align*}
If we denote the last three integrals for $I_1$, $I_2$ and $I_3$ respectively, we have
\begin{align*}
I_1&\leq \rho M \int_{-\infty}^\tau e^{-\gamma (\tau - r)} [(1+\Delta)\|v^\varepsilon-v^0\|_{X_\varepsilon^\frac{1}{2}}+|\!|\!|s_*^\varepsilon-s_*^0|\!|\!|]\,dr\\
&\leq \rho M(1+\Delta)\int_{-\infty}^\tau e^{-\gamma (\tau - r)}\|v^\varepsilon-v^0\|_{X_\varepsilon^\frac{1}{2}}\,dr +\rho M |\!|\!|s_*^\varepsilon-s_*^0|\!|\!|\int_{-\infty}^\tau e^{-\gamma (\tau - r)}\,dr\\
&= \rho M \gamma^{-1}|\!|\!|s_*^\varepsilon-s_*^0|\!|\!|+\rho M(1+\Delta)\int_{-\infty}^\tau e^{-\gamma (\tau - r)}\|v^\varepsilon-v^0\|_{X_\varepsilon^\frac{1}{2}}\,dr.
\end{align*}
For $I_2$ we have 
$$
G_\varepsilon(v^0,s_*^0(v^0))-G_0(v^0,s_*^0(v^0))=(Q_\varepsilon-Q_0)f(v^0+s_*^0(v^0)),
$$ 
and if we denote $\tau(\varepsilon)=(\|p_\varepsilon-p_0\|_{L^\infty(\Omega_1)}+\varepsilon)^\frac{1}{2}$, then $I_2\leq C\tau(\varepsilon)$. 
And for $I_3$, we have
$$
I_3\leq \int_{-\infty}^\tau l_\varepsilon(\tau-r)e^{-\gamma(\tau-r)}\,dr\leq C\tau(\varepsilon)|\log(\tau(\varepsilon))|,
$$
where we have used the Lemma 3.10 in \cite{Santamaria2017}. Thus
\begin{multline*}
\|s^\varepsilon_*(\eta)-s^0_*(\eta)\|_{X_\varepsilon^{\frac{1}{2}}} \leq C\tau(\varepsilon)|\log(\tau(\varepsilon))| \\ +\rho M \gamma^{-1} |\!|\!|s_*^\varepsilon-s^0_*|\!|\!| + \rho M(1+\Delta)\int_{-\infty}^\tau  e^{-\gamma(\tau-r)}\|v^\varepsilon-v^0\|_{X_\varepsilon^{\frac{1}{2}} }\,dr.
\end{multline*}
But, 

\begin{align*}
\|v^\varepsilon(t)-v^0(t)\|_{X_\varepsilon^{\frac{1}{2}}} & \leq \|(e^{-A_\varepsilon^+ (t-\tau)}-e^{-A_0^+ (t-\tau)})\eta \| \\
&+ \int_t^\tau \|e^{-A_\varepsilon^+ (t-r)}H_\varepsilon(v^\varepsilon,s_*^\varepsilon(v^\varepsilon))-e^{-A_\varepsilon^+ (t-r)}H_\varepsilon(v^0,s_*^0(v^0))\|_{X_\varepsilon^{\frac{1}{2}}}\,dr\\
&+ \int_t^\tau \|e^{-A_\varepsilon^+ (t-r)}H_\varepsilon(v^0,s_*^0(v^0))-e^{-A_\varepsilon^+ (t-r)}H_0(v^0,s_*^0(v^0))  \|_{X_\varepsilon^{\frac{1}{2}}}\,dr\\
&+ \int_t^\tau \|e^{-A_\varepsilon^+ (t-r)}H_0(v^0,s_*^0(v^0))-e^{-A_0^+ (t-r)}H_0(v^0,s_*^0(v^0))\|_{X_\varepsilon^{\frac{1}{2}}}\,dr
\end{align*}
With the same argument used earlier, we have 
\begin{multline*}
\|v^\varepsilon(t)-v^0(t)\|_{X_\varepsilon^{\frac{1}{2}}} \leq  C \int_{t}^\tau l_\epsilon(r-t) e^{\beta(t-r)}\,dr \\ +\rho M|\!|\!|s_*^\varepsilon-s^0_*|\!|\!|\int_t^\tau e^{\beta(t-r)} \,dr + \rho M(1+\Delta)\int_{t}^\tau  e^{\beta(t-r)}\|v^\varepsilon-v^0\|_{X_\varepsilon^{\frac{1}{2}} }\,dr.
\end{multline*}
By Gronwall's inequality,
$$
\|v^\varepsilon(t)-v^0(t)\|_{X_\varepsilon^{\frac{1}{2}}} \leq [C\tau(\varepsilon)|\log(\tau(\varepsilon))|+\rho M \beta^{-1}|\!|\!|s_*^\varepsilon-s^0_*|\!|\!|]e^{[\rho M(1+\Delta)-\beta](\tau-t)},
$$
thus
\begin{align*}
&\|s_*^\varepsilon(\eta)-s^0_*(\eta)\|_{X_\varepsilon^\frac{1}{2}}\leq  C\tau(\varepsilon)|\log(\tau(\varepsilon))| +\rho M \gamma^{-1} |\!|\!|s_*^\varepsilon-s^0_*|\!|\!|  \\
& + \rho M(1+\Delta)\int_{-\infty}^\tau e^{-\gamma(\tau-r)} e^{[\rho M(1+\Delta)-\beta](\tau - r)}\,dr[C\tau(\varepsilon)|\log(\tau(\varepsilon))|+\rho M \beta^{-1}|\!|\!|s_*^\varepsilon-s^0_*|\!|\!|] \\
& \leq C\tau(\varepsilon)|\log(\tau(\varepsilon))| +\Big[\rho M \gamma^{-1}+\frac{\rho^2 M^2(1+\Delta)\beta^{-1}}{\gamma+\beta-\rho M(1+\Delta)}\Big]|\!|\!|s_*^\varepsilon-s^0_*|\!|\!|
\end{align*}
which implies
$
|\!|\!|s_*^\varepsilon-s^0_*|\!|\!|\leq C\tau(\varepsilon)|\log(\tau(\varepsilon))|.
$

\textbf{Part 3}(Exponential attraction) It remains shown that $\mathcal{M}_\varepsilon$ is exponentially attracting and $\mathcal{A}_\varepsilon\subset\mathcal{M}_\varepsilon$. Let $(v^\varepsilon,z^\varepsilon)\in Y_\varepsilon\oplus Z_\varepsilon$ be the solution of \eqref{couple_system} and define $\xi^\varepsilon(t)=z^\varepsilon-s_*^\varepsilon(v^\varepsilon(t))$ and consider $y^\varepsilon(r,t), r\leq t$, $t\geq 0$, the solution through $r$ of 
$$
\begin{cases}
y_{r}^{\varepsilon}+A_\varepsilon^+ y^\varepsilon=H_\varepsilon(y^\varepsilon,s_*^\varepsilon(y^\varepsilon)),\quad  r\leq t \\
y^\varepsilon(t,t)=v^\varepsilon(t).
\end{cases}
$$
Thus,
\begin{align*}
\|y^\varepsilon(r,t)-&v^\varepsilon(r)\|_{X_\varepsilon^{\frac{1}{2}}} \\
& =\Big\|\int_t^r e^{-A_\varepsilon^+(r-\theta)}[H_\varepsilon(y^\varepsilon(\theta,t),s_*^\varepsilon(y^\varepsilon(\theta,t)))-H_\varepsilon(v^\varepsilon(\theta),z^\varepsilon
(\theta))]\,d\theta\Big\|_{X_\varepsilon^{\frac{1}{2}}} \\
&\leq \rho M \int_r^t e^{\beta(r-\theta)}[(1+\Delta)\|y^\varepsilon(\theta,t)-v^\varepsilon(\theta)\|_{X_\varepsilon^{\frac{1}{2}}}+\|\xi^\varepsilon(\theta)\|_{X_\varepsilon^{\frac{1}{2}}}]\,d\theta.
\end{align*}
By Gronwall's inequality
\begin{equation}\label{est-v-y}
\|y^\varepsilon(r,t)-v^\varepsilon(r)\|_{X_\varepsilon^{\frac{1}{2}}}\leq\rho M \int_r^{t} e^{-(\beta-\rho M(1+\Delta))(\theta-r)} \|\xi^\varepsilon(\theta)\|_{X_\varepsilon^{\frac{1}{2}}}\,d\theta\quad r\leq t.
\end{equation}
Now we take $t_0\in[r,t]$ and then
\begin{align*}
\|y^\varepsilon(r,t) &-y^\varepsilon(r,t_0)\|_{X_\varepsilon^{\frac{1}{2}}}\\ 
& = \|e^{-A_\varepsilon^+ (r-t_0)}[y(t_0,t)-v^\varepsilon(t_0)]\|_{X_\varepsilon^{\frac{1}{2}}} \\
& +\Big\| \int_{t_0}^r e^{-A_\varepsilon^+(r-\theta)}[H_\varepsilon(y^\varepsilon(\theta,t),s_*^\varepsilon(y^\varepsilon(\theta,t)))-H_\varepsilon(y^\varepsilon(\theta,t_0),s_*^\varepsilon(y^\varepsilon(\theta,t_0)))]\, d\theta\Big\|_{X_\varepsilon^{\frac{1}{2}}} \\
&\leq \rho M^2 e^{\beta (r-t_0)} \int_{t_0}^{t} e^{-(\beta-\rho M(1+\Delta))(\theta-t_0)}\|\xi^\varepsilon(\theta)\|_{X_\varepsilon^{\frac{1}{2}}}\,d\theta\\
&+\rho M\int_{r}^{t_0} e^{\beta(r-\theta)}(1+\Delta)\| y^\varepsilon(\theta,t)-y^\varepsilon(\theta,t_0)\|_{X_\varepsilon^{\frac{1}{2}}}\,d\theta. 
\end{align*}
By Gronwall's inequality
\begin{equation}\label{est-y-y}
\|y^\varepsilon(r,t)-y^\varepsilon(r,t_0)\|_{X_\varepsilon^{\frac{1}{2}}} \leq \rho M^2 \int_{t_0}^{t} e^{-(\beta-\rho M(1+\Delta))(\theta-r)} \|\xi^\varepsilon(\theta)\|_{X_\varepsilon^\frac{1}{2}}\,d\theta.
\end{equation}
In what follows we estimate $\xi^\varepsilon(t)$. Since
$$
z^\varepsilon(t)=e^{-A_\varepsilon^- (t-t_0)}z^\varepsilon(t_0)+\int_{t_0}^t e^{-A_\varepsilon^- (t-r)}G_\varepsilon(v^\varepsilon(r),z^\varepsilon(r))\,dr,
$$
we have
\begin{align*}
\xi^\varepsilon(t)- & e^{-A_\varepsilon^- (t-t_0)}\xi^\varepsilon(t_0) = z^\varepsilon(t)-s_*^\varepsilon(v^\varepsilon(t))- e^{-A_\varepsilon^- (t-t_0)}[z^\varepsilon(t_0)-s_*^\varepsilon(v^\varepsilon(t_0))] \\
&=\int_{t_0}^t e^{-A_\varepsilon^- (t-r)}G_\varepsilon(v^\varepsilon(r),z^\varepsilon(r))\,dr-s_*^\varepsilon(v^\varepsilon(t)) +e^{-A_\varepsilon^- (t-t_0)}s_*^\varepsilon(v^\varepsilon(t_0))\\
& =\int_{t_0}^t e^{-A_\varepsilon^- (t-r)}G_\varepsilon(v^\varepsilon(r),z^\varepsilon(r))\,dr -\int_{-\infty}^t e^{-A_\varepsilon^- (t-r)}G_\varepsilon(y^\varepsilon(r,t),s_*^\varepsilon(y^\varepsilon(r,t)))\,dr\\
& +e^{-A_\varepsilon^- (t-t_0)}\int_{-\infty}^{t_0} e^{-A_\varepsilon^- (t_0-r)}G_\varepsilon(y^\varepsilon(r,t_0),s_*^\varepsilon(y^\varepsilon(r,t_0)))\,dr\\
& =\int_{t_0}^t e^{-A_\varepsilon^- (t-r)}[G_\varepsilon(v^\varepsilon(r),z^\varepsilon(r))-G_\varepsilon(y^\varepsilon(r,t),s_*^\varepsilon(y^\varepsilon(r,t))) ]\,dr\\
& - \int_{-\infty}^{t_0} e^{-A_\varepsilon^- (t-r)}[ G_\varepsilon(y^\varepsilon(r,t),s_*^\varepsilon(y^\varepsilon(r,t)))  -G_\varepsilon(y^\varepsilon(r,t_0),s_*^\varepsilon(y^\varepsilon(r,t_0)))]\,dr.
\end{align*}
Thus, using \eqref{est-v-y} and \eqref{est-y-y},
\begin{align*}
\|&\xi^\varepsilon(t) -  e^{-A_\varepsilon^- (t-t_0)}\xi^\varepsilon(t_0) \|_{X_\varepsilon^{\frac{1}{2}}} \\
&\leq \rho M \int_{t_0}^t e^{-\gamma (t-r)}[\|v^\varepsilon(r)-y^\varepsilon(r,t)\|_{X_\varepsilon^{\frac{1}{2}}}+\|z^\varepsilon(r)-s_*^\varepsilon(y^\varepsilon(r,t))\|_{X_\varepsilon^{\frac{1}{2}}}]\,dr\\
& + \rho M(1+\Delta)\int_{-\infty}^{t_0}e^{-\gamma (t-r)}\|y^\varepsilon(r,t)-y^\varepsilon(r,t_0)\|_{X_\varepsilon^{\frac{1}{2}}}\,dr\\
&\leq \rho M \int_{t_0}^t e^{-\gamma (t-r)}\|\xi^\varepsilon(r)\|_{X_\varepsilon^{\frac{1}{2}}}\,dr\\
& +\rho^2 M^2(1+\Delta)e^{-\gamma t}\int_{t_0}^t e^{-(\beta-\rho M(1+\Delta))\theta} \|\xi^\varepsilon(\theta)\|_{X_\varepsilon^{\frac{1}{2}}}\int_{-\infty}^{\theta} e^{(\gamma+\beta-\rho M(1+\Delta) )r}\,dr d\theta  \\
& +\rho^2 M^3(1+\Delta)e^{-\gamma t}\int_{t_0}^t e^{-(\beta -\rho M(1+\Delta))\theta} \|\xi^\varepsilon(\theta)\|_{X_\varepsilon^{\frac{1}{2}}} \int_{-\infty}^{t_0} e^{(\gamma+\beta-\rho M(1+\Delta) )r}\,dr d\theta, 
\end{align*}
and then
\begin{align*}
\|\xi^\varepsilon(t) & -  e^{-A_\varepsilon^- (t-t_0)}\xi^\varepsilon(t_0) \|_{X_\varepsilon^{\frac{1}{2}}} \leq [\rho M-\frac{\rho^2 M^2(1+\Delta)}{\gamma+\beta -\rho M(1+\Delta) }\Big]\int_{t_0}^t e^{-\gamma(t-\theta)}\|\xi^\varepsilon(\theta)\|_{X_\varepsilon^{\frac{1}{2}}} \,d\theta\\
& + \frac{\rho^2M^3(1+\Delta) e^{-\gamma (t-t_0)}}{\gamma+\beta-\rho M (1+\Delta)}\int_{t_0}^t e^{-(\beta-\rho M(1+\Delta){\color{red})}(\theta-t_0)}\|\xi^\varepsilon(\theta)\|_{X_\varepsilon^{\frac{1}{2}}}\,d\theta.
\end{align*}
Hence
\begin{align*}
e^{\gamma (t-t_0)} \|\xi^\varepsilon(t)\|_{X_\varepsilon^{\frac{1}{2}}} & \leq M \|\xi^\varepsilon(t_0)\|_{X_\varepsilon^{\frac{1}{2}}}+ \Big[\rho M+\frac{\rho^2 M^2(1+\Delta)}{\gamma+\beta-\rho M(1+\Delta) }\Big]\int_{t_0}^t e^{\gamma(r-t_0)}\|\xi^\varepsilon(r)\|_{X_\varepsilon^{\frac{1}{2}}} \,dr\\
& +\frac{\rho^2M^3(1+\Delta)}{\gamma-\beta-\rho M (1+\Delta)}\int_{t_0}^t e^{-(\gamma+\beta-\rho M(1+\Delta){\color{red})}(\theta-t_0)}e^{\beta (\theta-t_0)}\|\xi^\varepsilon(\theta)\|_{X_\varepsilon^{\frac{1}{2}}}\,d\theta\\
&\leq M \|\xi^\varepsilon(t_0)\|_{X_\varepsilon^{\frac{1}{2}}}+\Big[\rho M+\frac{\rho^2 M^2(1+\Delta)(1+M)}{\gamma+\beta-\rho M(1+\Delta)}\Big]\int_{t_0}^t e^{\gamma(r-t_0)}\|\xi^\varepsilon(r)\|_{X_\varepsilon^{\frac{1}{2}}}\,dr.
\end{align*}
By Gronwall's inequality
$$
\|\xi^\varepsilon(t)\|_{X_\varepsilon^{\frac{1}{2}}}\leq M\|\xi^\varepsilon(t_0)\|_{X_\varepsilon^{\frac{1}{2}}}e^{-(\gamma-L)(t-t_0)},
$$
and then
$$
\|z^\varepsilon(t)-s_*^\varepsilon(v^\varepsilon(t))\|_{X_\varepsilon^{\frac{1}{2}}}=\|\xi^\varepsilon(t)\|_{X_\varepsilon^{\frac{1}{2}}}\leq M\|\xi^\varepsilon(t_0)\|_{X_\varepsilon^{\frac{1}{2}}}e^{-(\gamma-L)(t-t_0)}.
$$

Now if $u^\varepsilon:=T_\varepsilon(t)u_0^\varepsilon=v^\varepsilon(t)+z^\varepsilon(t)$, $t\in\R$, denotes the solution through at $u_0^\varepsilon=v_0^\varepsilon+z_0^\varepsilon\in \mathcal{A}_\varepsilon$, then
$$
\|z^\varepsilon(t)-s_*^\varepsilon(v^\varepsilon(t))\|_{X_\varepsilon^{\frac{1}{2}}}\leq M\|z_0^\varepsilon-s_*^\varepsilon(v_0^\varepsilon)\|_{X_\varepsilon^{\frac{1}{2}}}e^{-(\gamma-L)(t-t_0)}.
$$
Since $\{T_\varepsilon(t)u_0^\varepsilon\,;\,t\in\R\}\subset\mathcal{A}_\varepsilon$ is bounded, letting $t_0\to-\infty$ we obtain $T_\varepsilon(t)u_0^\varepsilon=v^\varepsilon(t)+s_*^\varepsilon(v^\varepsilon(t))\in\mathcal{M}_\varepsilon$. That is $\mathcal{A}_\varepsilon\subset\mathcal{M}_\varepsilon$. Moreover, if $B_\varepsilon\subset X_\varepsilon^\frac{1}{2}$ is a bounded set and  $u_0^\varepsilon=v_0^\varepsilon+z_0^\varepsilon\in B_\varepsilon$, we conclude that $T_\varepsilon(t)u_0^\varepsilon=v^\varepsilon(t)+z^\varepsilon(t)$ satisfies
\begin{align*}
\sup_{u_0^\varepsilon\in B_\varepsilon}\inf_{w\in\mathcal{M}_\varepsilon}\|T_\varepsilon(t)u_0^\varepsilon-w\|_{X_\varepsilon^\frac{1}{2}}&\leq \sup_{u_0^\varepsilon\in B_\varepsilon}\|z^\varepsilon(t)-s_*^\varepsilon(v^\varepsilon(t))\|_{X_\varepsilon^\frac{1}{2}}\\
&\leq Me^{-(\gamma-L)(t-t_0)} \sup_{u_0^\varepsilon\in B_\varepsilon}\|z_0^\varepsilon-s_*^\varepsilon(v_0^\varepsilon)\|_{X_\varepsilon^{\frac{1}{2}}},
\end{align*}
which implies
$$
\tn{dist}_H(T_\varepsilon(t)B_\varepsilon,\mathcal{M}_\varepsilon)\leq C(B_\varepsilon)e^{-(\gamma-L)(t-t_0)},
$$
and thus the proof is complete.
\end{proof}

\begin{rem}\label{C_convergence}
It is well known in the theory of invariant manifolds the $C^0$, $C^1$ and $C^{1,\theta}$ convergences of invariant manifolds (see \cite{Carbone2006} and \cite{Santamaria2017}). That is  
$$
\|s_\ast^\varepsilon-s_\ast^0\|_{C^0(Y_0)}, \|s_\ast^\varepsilon-s_\ast^0\|_{C^1(Y_0)}, \|s_\ast^\varepsilon-s_\ast^0\|_{C^{1,\theta}(Y_0)}\overset{\varepsilon\to 0} \longrightarrow 0.
$$
\end{rem}

\section{Rate of Convergence of Attractors}\label{Rate of Convergence of Attractors}
In this section we will estimate the continuity of attractors of \eqref{semilinear_problem} in the Hausdorff metric by the rate of convergence of resolvent operators obtained in the Section \ref{Rate of Convergence of Eigenvalues and Equilibria}.

The operator $A_\varepsilon$, $\varepsilon\in[0,\varepsilon_0]$, has compact resolvent and according to \cite{Carbone2006}, $A_0$ is Sturm-Liouville type, which implies transversality of stable and unstable manifolds of the equilibrium points. Since we assume hyperbolicity, the limiting problem \eqref{limite_problem} generates a Morse-Smale semigroup in $X_0^\frac{1}{2}$ and hence the perturbed problem \eqref{perturbed_problem} generates a Morse-Smale semigroup in $X_\varepsilon^\frac{1}{2}$. 

We saw in the last section how the gap condition implies the existence of the finite dimensional invariant manifold $\mathcal{M}_\varepsilon$ for \eqref{semilinear_problem}. The invariant manifold contains the attractor $\mathcal{A}_\varepsilon$ and the flow is given by an ordinary differential equation. That is,
$$
u^\varepsilon(t)=v^\varepsilon(t)+s_\ast^\varepsilon(v^\varepsilon(t)), \quad t\in\R,
$$ 
where $v^\varepsilon(t)$ satisfies
\begin{equation}\label{eq_edo_reduced}
\dot{v^\varepsilon}+A_\varepsilon^+v^\varepsilon=Q_\varepsilon f(v^\varepsilon+ s_\ast^\varepsilon(v^\varepsilon(t))),
\end{equation}
and we can consider $v^\varepsilon\in \R^m$ and $H_\varepsilon(v^\varepsilon)=Q_\varepsilon f(v^\varepsilon+ s_\ast^\varepsilon(v^\varepsilon(t)))$ a continuously differentiable map in $\R^m$. For each $\varepsilon\in [0,\varepsilon_0]$, we denote $\tilde{T}_\varepsilon=\tilde{T}_\varepsilon(1)$, where $\tilde{T}_\varepsilon(\cdot)$ is the semigroup generated by solution $ v^\varepsilon(\cdot)$  of \eqref{eq_edo_reduced} in $\R^m$. We have the following convergences
\begin{equation}\label{estimate_one_map}
\|\tilde{T}_\varepsilon-\tilde{T}_0\|_{C^1(\R^m,\R^m)}\overset{\varepsilon\to 0}\longrightarrow 0\quad\tn{and}\quad\|\tilde{T}_\varepsilon-\tilde{T}_0\|_{L^\infty(\R^m,\R^m)}\leq C\tau(\varepsilon)|\log(\tau(\varepsilon))|,
\end{equation}
where $\tau(\varepsilon)=(\|p_\varepsilon-p_0\|_{L^\infty(\Omega_1)}+\varepsilon)^\frac{1}{2}$ and the last estimate is proved as Theorem \ref{continuity_nonlinear_semigroup}. 

Since we have a Morse-Smale semigroup in $\R^m$, by using techniques of shadowing in \cite{Pilyugun1999}, we have the following result that was proved in \cite{Santamaria2017}.

\begin{theo}[\cite{Santamaria2017}]\label{Shadowing}
Let $T:\R^m\to \R^m$ be a discrete Morse-Smale semigroup with a global attractor $\mathcal{A}$. Then there are a positive constant $L$, a neighborhood $\mathcal{N}(\mathcal{A})$ of $\mathcal{A}$ and a neighborhood $\mathcal{N}(T)$ of $T$ in the $C^1(\mathcal{N}(\mathcal{A}),\R^m)$ topology such that, for any $T_1,T_2\in \mathcal{N}(T)$ with attractors  $\mathcal{A}_1$, $\mathcal{A}_2$ respectively, we have
$$
\tn{dist}_H(\mathcal{A}_1,\mathcal{A}_2)\leq L\|T_1-T_2\|_{L^\infty(\mathcal{N}(\mathcal{A}),\R^m)}.
$$
\end{theo}

Now we are ready to prove the main result of this paper.

\begin{theo} Let $\mathcal{A}_\varepsilon$, $\varepsilon\in[0,\varepsilon_0]$, be the attractor for \eqref{semilinear_problem}. Then there is a positive constant $C$ independent of $\varepsilon$ such that
$$
\tn{d}_H(\mathcal{A}_\varepsilon,\mathcal{A}_0)\leq C(\|p_\varepsilon-p_0\|_{L^\infty(\Omega_1)}+\varepsilon)^\frac{1}{2}|\log(\|p_\varepsilon-p_0\|_{L^\infty(\Omega_1)}+\varepsilon)|.
$$
\end{theo}
\begin{proof} We will follow \cite{Santamaria2017}. For each $\varepsilon\in [0,\varepsilon_0]$ we denote $T_\varepsilon=T_\varepsilon(1)$. Given $u^\varepsilon\in \mathcal{A}_\varepsilon$, by invariance there is $w^\varepsilon\in \mathcal{A}_\varepsilon$ such that $u^\varepsilon=T_\varepsilon w^\varepsilon$ so we can write 
$
w^\varepsilon=Q_\varepsilon w^\varepsilon + s_\ast^\varepsilon(Q_\varepsilon w^\varepsilon),
$
where $Q_\varepsilon w^\varepsilon \in \bar{\mathcal{A}_\varepsilon}$ with $\bar{\mathcal{A}_\varepsilon}=Q_\varepsilon\mathcal{A}_\varepsilon$ the projected attractor in $\R^m$. Thus
\begin{align*}
\|u^\varepsilon-u^0\|_{X_\varepsilon^\frac{1}{2}} &=\|T_\varepsilon w^\varepsilon-T_0 w^0\|_{X_\varepsilon^\frac{1}{2}}\leq \|T_\varepsilon w^\varepsilon-T_\varepsilon w^0\|_{X_\varepsilon^\frac{1}{2}}+\|T_\varepsilon w^0-T_0 w^0\|_{X_\varepsilon^\frac{1}{2}}\\
& \leq C(\|p_\varepsilon-p_0\|_{L^\infty(\Omega_1)}+\varepsilon)^\frac{1}{2}|\log(\|p_\varepsilon-p_0\|_{L^\infty(\Omega_1)}+\varepsilon)| +C\|w^\varepsilon-w^0\|. 
\end{align*}
But 
\begin{align*}
\|w^\varepsilon-w^0\|&=\|Q_\varepsilon w^\varepsilon-Q_0 w^0\|_{\R^m}+\|s_\ast^\varepsilon(Q_\varepsilon w^\varepsilon)-s_\ast^0(Q_0 w^0)\|_{X_\varepsilon^\frac{1}{2}}\\
&\leq \|Q_\varepsilon w^\varepsilon-Q_0 w^0\|_{\R^m}+\|s_\ast^\varepsilon(Q_\varepsilon w^\varepsilon)-s_\ast^\varepsilon(Q_0 w^0)\|_{X_\varepsilon^\frac{1}{2}}+\|s_\ast^\varepsilon(Q_0 w^0)-s_\ast^0(Q_0 w^0)\|_{X_\varepsilon^\frac{1}{2}}\\
&\leq C \|Q_\varepsilon w^\varepsilon-Q_0 w^0\|_{\R^m}+C(\|p_\varepsilon-p_0\|_{L^\infty(\Omega_1)}+\varepsilon)^\frac{1}{2}|\log(\|p_\varepsilon-p_0\|_{L^\infty(\Omega_1)}+\varepsilon)|, 
\end{align*}
which implies 
$$
\tn{d}_H(\mathcal{A}_\varepsilon,\mathcal{A}_0)\leq \tn{d}_H(\bar{\mathcal{A}_\varepsilon},\bar{\mathcal{A}_0})+C(\|p_\varepsilon-p_0\|_{L^\infty(\Omega_1)}+\varepsilon)^\frac{1}{2}|\log(\|p_\varepsilon-p_0\|_{L^\infty(\Omega_1)}+\varepsilon)|.
$$
The result follows by \eqref{estimate_one_map} and Theorem \ref{Shadowing}. 
\end{proof}

\bibliographystyle{abbrv}
\bibliography{References}


\end{document}